\documentclass[11pt, reqno]{amsart}
\usepackage[bookmarks]{hyperref}
\usepackage{amsmath,amsfonts,mathrsfs,amsthm}
\usepackage{color}

\newtheorem{thm}{Theorem}[section]

\newtheorem{cor}[thm]{Corollary}
\newtheorem{lem}[thm]{Lemma}

\theoremstyle{definition}
\newtheorem{remark}{Remark}

\numberwithin{equation}{section}

\usepackage[margin = 0.5in]{geometry}
\usepackage[skip = 10pt,indent = 0pt]{parskip}

\newcommand{\real}{\mathbb{R}}
\newcommand{\integ}{\mathbb{Z}}

\DeclareMathOperator{\Vol}{Vol}

\newcommand{\dd} {\mathrm{d}}

\DeclareMathOperator{\supp}{supp}
\DeclareMathOperator{\dv}{div}

\makeatletter
\newcommand{\labelt}[1]{%
  \phantomsection
  \renewcommand{\@currentlabel}{#1}
  \label{#1}
}
\makeatother

\title[the divergence equation]{Real Analytic solutions to the divergence equation}

\author[Chan]{Chi Hin Chan}
\address{Department of Applied Mathematics, National Yang Ming Chiao Tung University, No. 1001, Daxue Rd. East Dist., Hsinchu City 300093, Taiwan}
\email{cchan@math.nctu.edu.tw}

\author[Chen]{Jun-Shuo Chen}
\address{Department of Mathematics,
National Tsing Hua University\\ 101, Section 2, Kuang-Fu Road, Hsinchu 300044, Taiwan}
\email{jschen96@gmail.com}

\author[Su]{Cheng-Fang Su}
\address{Department of Applied Mathematics, National Yang Ming Chiao Tung University, No. 1001, Daxue Rd. East Dist., Hsinchu City 300093, Taiwan}
\email{scf1204@nycu.edu.tw}

\begin{document}
\begin{abstract}
In this paper, we develop a differential-topological method to yield explicit real analytic solutions $v$ to the divergence equation $\dv_{\mathbb{R}^n} v = f$ on any annali $A(R_1 ,R_2) = \{ x \in \mathbb{R}^n : R_1 < |x| <  R_2\}$, with $n \geq 2$, and  $0 < R_1 < R_2 < \infty$. The prescribed source term $f$ is supposed to be real analytic on $\overline{A(R_1 , R_2)} = \{ x \in \mathbb{R}^n : R_1 \leq |x| \leq  R_2\}$ satisfying $\int_{A(R_1 , R_2)} f \Vol_{\mathbb{R}^n} = 0$. The resulting solution $v$ is a real analytic vector field on $\overline{A(R_1 , R_2)}$, which vanishes on $\partial \big(  A(R_1, R_2 ) \big )$. The method which we develop here is  different from the standard Bogovski approach \cite{Bogovski1979} and the Kapitanskii-Pileckas approach \cite{KP}. The first main step our method is a clever differential-topological argument, which we develop under the inspiration and guidance of the standard proof of the cohomological statement $H_c^n \big ( \mathbb{R}^n\big ) = \mathbb{R}$ in Spviak book \cite{Spivak}\emph{A Comprehensive Introduction to Differential Geometry, Vol I}. This allows us to reduce the problem to that of solving a linear algebra problem.
\end{abstract}
\maketitle

\section{Introduction of our paper.}
The aim of this paper is to construct explicit real analytic solutions to the divergence equation $\dv_{\mathbb{R}^n} \textbf{U} = f$ on a region $A(R_1 , R_2 ) = \{ x \in \mathbb{R}^n : R_1 < |x| < R_2  \}$, where the real analyticity of the prescribed $f$ and the real analyticity of the solution $\textbf{U}$ which we seek both are supposed to be on $\overline{A(R_1, R_2)}$. The main result of our paper is stated in Theorem \ref{LooseForm} and Theorem \ref{TheMainTheoremWehave}.

\begin{thm}\label{LooseForm}
Consider $n \in \mathbb{Z}^+$ with $n \geq 2$, and $0< R_1 < R_2 < \infty$. Let $f : \overline{A(R_1, R_2)}\rightarrow \mathbb{R}$ be real analytic on $\overline{A(R_1 , R_2)}$ and satisfies $\int_{\overline{A(R_{1},R_2)}} f \Vol_{\mathbb{R}^n} = 0$. Then, it follows that there exists a vector field $\textbf{U} : \overline{A(R_1 , R_2)} \rightarrow \mathbb{R}^3$, which is real analytic on $\overline{A(R_1 ,R_2)}$, such that
\begin{equation}\label{plaineasy1}
\dv_{\mathbb{R}^n} \textbf{U} (x) = f(x) ,
\end{equation}
holds for all $x \in \overline{A(R_1 , R_2 )}$, and that
\begin{equation}\label{plaineasy2}
\textbf{U} \big |_{\partial (A (R_1 , R_2))} = 0.
\end{equation}
\end{thm}

\begin{remark}
In general, we say that a function $f : \overline{A(R_1, R_2)} \rightarrow \mathbb{R}$ is real analytic on $\overline{A(R_1, R_2 )}$ if there exists some $\delta \in (0,1)$, together with some real analytic function $\widetilde{f} : A(R_1 (1-\delta ) , R_2 (1 + \delta )) \rightarrow \mathbb{R}$ on $A(R_1 (1-\delta ) , R_2 (1 + \delta )) $ such that we have 
\begin{equation*}
\widetilde{f} \big |_{\overline{A(R_1, R_2 )}} = f .
\end{equation*}
This concept of real analyticity over $\overline{A(R_1, R_2)}$ clearly also applies to $\mathbb{R}^n$-valued function on $\overline{A(R_1, R_2)}$. Thus, it makes sense to talk about real analytic vector fields on $\overline{A(R_1, R_2)} $.
\end{remark}

Theorem \ref{LooseForm} only provides a loose form of the actual result, namely Theorem \ref{TheMainTheoremWehave}, of this paper $($see Remark \ref{firstremark} $)$. Theorem \ref{TheMainTheoremWehave} is the one which contains all the precise details useful for the analysis of PDE. Due to the technical nature of Theorem \ref{TheMainTheoremWehave}, we first spell out the essense of Theorem \ref{TheMainTheoremWehave} in a more readable form of Theorem \ref{LooseForm}. A reader who takes a strong interest in the analytical details of our work should jump directly to the statement of Theorem \ref{TheMainTheoremWehave}.

\begin{remark}\label{firstremark}
In the standard framework of $f \in L^{p} (\Omega )$, with $1 < p < \infty$ and $\int_{\Omega} f \Vol_{\mathbb{R}^n} = 0$, where $\Omega$ is assumed to be a bounded Lipschitz domain in $\mathbb{R}^n$ with $n \geq 2$, a solution $\textbf{U} \in W_0^{1,p} (\Omega )$ to \eqref{plaineasy1} and \eqref{plaineasy2} is in general not unique, and the standard Bogovski's map approach \cite{Bogovski1979} only provides one possible $W^{1,p}_0(\Omega )$-solutions among many others. Indeed, this comment applies equally here, in regard to real analytic solutions $\textbf{U} : \overline{A(R_1 , R_2)} \rightarrow \mathbb{R}^n$ to \eqref{plaineasy1} and \eqref{plaineasy2}. As such, it is important to give explicit construction of a real analytic solution to \eqref{plaineasy1} and \eqref{plaineasy2} in a particular way.
Theorem \ref{TheMainTheoremWehave} does provide such a specific construction. Thus Theorem \ref{TheMainTheoremWehave} is much stronger than that of Theorem \ref{LooseForm}, and Theorem \ref{LooseForm} is the loose form of the more precise Theorem \ref{TheMainTheoremWehave}.
\end{remark}

Before we can comment on the method which we used in the establishment of Theorem \ref{LooseForm} and Theorem \ref{TheMainTheoremWehave}, we have to mention two widely used and very well-known standard techniques for constructing solutions of the divergence equation $\dv_{\mathbb{R}^n} \textbf{W} = f$.

The first one is the Bogovski map approach as developed by Bogovski \cite{Bogovski1979} within the framework of $W_0^{1,p}(\Omega )$-Sobolev space, in which $\Omega$ is a bounded Lipschitz domain in $\mathbb{R}^n$ $(n \geq 2 )$, $1 < p < \infty$, and $f \in L^p (\Omega )$ satisfies $\int_{\Omega} f \Vol_{\mathbb{R}^n} = 0$. The Bogovski map assigns to each $f \in L^{p}(\Omega)$ with $\int_{\Omega} f \Vol_{\mathbb{R}^n} =0$ a vector field $\textbf{W} = Bog_{\Omega}(f) \in W_0^{1,p} (\Omega )$ solving the equation $\dv_{\mathbb{R}^n} \textbf{W} = f$ on $\Omega$ in the weak sense, and that $\textbf{W} = Bog_{\Omega}(f)$ satisfies the following apriori estimate, with $C(N, p, \Omega )$ be a scale-invariant absolute constant which depends only on $N$, $p$ and the shape of $\Omega$.
\begin{equation*}
\big \| \nabla \textbf{W} \big \|_{L^p (\Omega )} \leq C(N ,p, \Omega ) \|f\|_{L^p (\Omega )}.
\end{equation*}

The second one is the well-known Kapitanskii-Pileckas approach $($see Kapitanskii-Pileckas \cite{KP}, as well as \cite{Dacorogna2002}$)$. Kapitanskii-Pileckas approach works for the setting of $f \in C^{l, \alpha} ( \overline{\Omega} )$ with $l \in \mathbb{N}$, $0 < \alpha < 1$, and $\int_{\Omega} f \Vol_{\mathbb{R}^n} =0$, where $\Omega$ is a bound domain in $\mathbb{R}^n$ with $\partial \Omega$ be a $C^{l+2 , \alpha }$-hypersurface in $\mathbb{R}^n$. The standard Kapitanskii-Pileckas approach provides a vector field $\textbf{W} \in C^{l+1, \alpha} (\overline{\Omega })$ which solves $\dv_{\mathbb{R}^n} \textbf{W} = f$ on $\overline{\Omega}$ and satisfies $\textbf{W} \big |_{\partial \Omega } = 0$. Moreover, such a $\textbf{W}$ satisfies the following apriori estiamte with an absolute constant $C (N, l ,\alpha , \Omega ) > 0$.
\begin{equation*}
\|\textbf{W}\|_{C^{l+1} (\overline{\Omega})} \leq C (N, l ,\alpha , \Omega ) \|f\|_{C^{l, \alpha }(\overline{\Omega})} .
\end{equation*}

At this point, we want to point out that the approach which we develop here for the proof of Theorem \ref{LooseForm} and Theorem \ref{TheMainTheoremWehave} is {different from} the standard Bogovski map approach \cite{Bogovski1979} and the Kapitanskii-Pileckas approach \cite{KP} mentioned above. We instead take the view point of De Rham Cohomology with compact support, and in practice:

\begin{itemize}
\item The first major step which we take here, thus leading us to the successful establishment of Lemma \ref{keylemma} and Corollary \ref{corollaryfromCohomology}, is basically inspired and guided by Spivak's proof of the cohomological statement $H_c^n(\mathbb{R}^n ) = \mathbb{R}$ in Chapter 8 of his book: \emph{A comprehensive Introduction to Differential Geometry, Volume I} $($ see page364-page367, Theorem 9, Chapter 8 of \cite{Spivak} $)$.
\end{itemize}
Besides the above mentioned influence of Spivak's proof of Theorem 9 of Chapter 8 in \cite{Spivak} upon us for our development of Lemma \ref{keylemma} and Corollary \ref{corollaryfromCohomology} , our method is basically original and self-contained.
\begin{itemize}
\item As such, in {section \ref{Appendix}} , we give a careful comparison between Lemma \ref{keylemma} which we developed with the original Spivak's proof of Theorem 9 of Chapter 8 in \cite{Spivak}. {Remark \ref{rmk:lastRemark}} explains precisely the new contributions which we input into the statement and proof of Lemma \ref{keylemma}.
\end{itemize}

Due to the fact that the first main step $($ Lemma \ref{keylemma} and Corollary \ref{corollaryfromCohomology} $)$ which we develop in this paper has its origin in the theory of cohomology with compact support, we also need to mention that we are definitely not the first to be aware of the deep connection between De Rham Cohomology theory and solutions to the divergence equation.

Indeed, as early as 1992, S. Takahashi \cite{Takahashi} extend the original idea of Bogovski \cite{Bogovski1979} to setting of differential forms on a Euclidean region which is star shaped with respect to every point of an open ball. Along this line of research, next we see the very influential work of M. Mitrea \cite{Mitrea2004}, in which M. Mitrea studied, among other things, generalized integral operators of Bogovski-type and regularized Poincare-type in the setting of differential forms with coefficient functions lying in Besov or Triebel-
Lizorkin spaces $($see the introduction of recent work \cite{Costabel} by Costabel and McIntosh for the precise descriptions of generalized Bogovski-type and Poincare-type integral operators in the setting of differential forms, as well as that of Mitrea's work \cite{Mitrea2004}  $)$. The work \cite{Costabel} appears to be one of the most recent works which can reflect the most recent developments of this subject matter and the reader can consult \cite{Costabel} for further references.

Before we say a few words about the way in which our paper is structured and organized, we would like to say that:
\begin{itemize}
\item In spite of the vast literature devoted to the study of the divergence equation $\dv_{\mathbb{R}^n} v = f$ under various regularity assumptions imposed on $f$ and the boundary of the domain region, the highest possible regularity assumption which we could see in the existing literature seems to be that of $f \in C_c^{\infty}(\Omega)$ or $f \in C^{\infty} (\Omega)$.
\end{itemize}
So, to the best of our knowledge, the study of the divergence equation $\dv_{\mathbb{R}^n} v = f$ under the Dirichlet condition $v \big |_{\partial \Omega}=0$ in the real analytic setting is seldomly seen. Thus, our establishment of Theorem \ref{LooseForm} and Theorem \ref{TheMainTheoremWehave} can be viewed as a very modest attempt along this direction.

Finally but not the least, we need to say that literature devoted to the study of the divergence equation in the framework of Sobolev spaces and that of Holder spaces is vast, and we do not attempt to do a survey of all the important research works here. However, a very detailed and accurate survey of many important and well-known literatures and research works for this subject is available in section III.7, page 227, of Chapter III the book \cite{Galdi} by Galdi.

\newpage
\section{Preparations}
In this section, we present several analytical tools that will be employed in the proofs of Lemma \ref{keylemma} and Theorem \ref{TheMainTheoremWehave}. The proof of Lemma \ref{lem:S.2} is standard and we leave to the readers. 

We use Lemma \ref{lem:S.2} to provide a proper logical support for the establishment of Lemma \ref{lem:s.3}. While, we use Lemma \ref{lem:s.3} to provide logical support for the establishment of Lemma \ref{keylemma}.

On the other hand, the establishment of Lemma \ref{lem:s.4} is logically independent from that of Lemma \ref{lem:S.2}, Lemma \ref{lem:s.3} and Lemma \ref{keylemma}. Instead, we use Lemma \ref{lemma5.1} to support the establishment of Lemma \ref{lem:s.4}.

Finally, we use Lemma \ref{lem:s.4} and Lemma \ref{keylemma} to obtain Theorem \ref{mainthm}, as well as Theorem \ref{LooseForm}, which are the main results of this paper.

\begin{lem}\label{lem:S.2}
    Take $n \in \mathbb{Z}^{+}$ which satisfies $n\geq 2$. Consider three real numbers $-\infty < L_{1} < L_{2} < L_{3} < +\infty$. Let $U$ be an open set in $\real^{n-1}$. Let $h:U \times (L_{1},L_{3}) \to \real^{1}$ be real analytic on $U \times (L_{1},L_{3})$ for which we have \begin{equation}\label{eq:s.2.1}
        h|_{U \times \{L_{2}\}} \equiv 0. 
    \end{equation}
    Then it follows that there exists a unique real-analytic function $h_{1}:U \times (L_{1},L_{3}) \to \real^{1}$ for which the following relation holds:
    \begin{equation}\label{eq:s.2.2}
        h(y) = (y_n - L_{2})\cdot h_{1}(y) \qquad \text{for all $y \in U \times (L_{1},L_{3})$}
    \end{equation}
\end{lem}

\begin{lem}\label{lem:s.3}
    Consider $n \in \mathbb{Z}^{+}$ which satisfies $n \geq 2$ and $l = 0$ or $1$. Let $U$ be an open set in $\mathbb{R}^{n}$. Take $ 0 < R_{1} < R_{2} < \infty$. For any scalar-valued function $G \in C^{l}(U \times (0, +\infty))$ satisfying the following two properties:
    \begin{description}
        \item[Property 1] $G|_{U \times [R_{1},R_{2}]}$ is real-analytic on $U \times [R_{1},R_{2}]$,
        \item[Property 2] $\partial_{n}G = {\partial G \over \partial y_{n}} \in C^{l}(U \times(0,+\infty))$,
    \end{description}
    The following two statements are valid:
    \begin{description}
        \item[Statement 1] If $G$ further satisfies $G \equiv 0$ on $U \times (0,R_{1}]$, then it follows that $G \in C^{l+1}(U \times (0,R_{2}))$.
        \item[Statement 2] If $G$ further satisfies $G \equiv 0$ on $U \times [R_{2},+\infty)$, then it follows that $G \in C^{l+1}(U\times (R_{1},+ \infty))$.
    \end{description}
\end{lem}

\begin{proof}
    It suffices to just prove {Statement 1} since the proof of {Statement 2} is basically identical. Furthermore, we present the proof only for the case $l = 1$. An analogous argument applies to the case $l = 0$.
    Let $U$ be an open set in $\real^{n-1}$. A typical point in $U$ will be denoted by $\bar{y} = (y_1, \dots, y_{n-1}) \in U$, so that a typical point in $U \times (0, +\infty)$ will be denoted by $y = (\bar{y},y_{n}) = (y_{1}, \dots, y_{n-1},y_{n}) \in U \times (0, +\infty)$.

Now, suppose that 
    \begin{equation}\label{eq:s.1.14}
        G \in C^{1}(U \times (0, +\infty))
    \end{equation}
    further satisfies the following hypothesis: for some $0 < R_{1} < R_{2} < + \infty$, we have
    \begin{description}
        \item[hypothesis 1] \label{eq:s.1.15} $G|_{U \times[R_{1},R_{2}]}$ is real analytic on $U \times [R_{1},R_{2}]$.
        \item[hypothesis 2] \label{eq:s.1.16} $\partial_{n}G = {\partial G \over \partial y_{n}} \in C^{1}(U \times(0,+\infty))$
        \item[hypothesis 3] \label{eq:s.1.17} $G \equiv 0$ on $U \times (0,R_{1}]$
    \end{description}
    Based on \eqref{eq:s.1.14}, {hypothesis 1, 2, 3}, we have to prove that \begin{equation}\label{eq:s.1.18}G \in C^{2}(U \times (0, R_{2})).\end{equation}

    Now {hypothesis 1} implies that there exists some $\delta_0 \in (0,1)$ for which we can find some real-analytic function $\tilde{G}:U \times (R_{1}(1- \delta_{0}),R_{2}(1+\delta_0)) \to \real^1$ such that \begin{equation} \label{eq:s.1.19}
        \tilde{G}|_{U \times [R_{1},R_{2}]} = G|_{U \times [R_{1},R_{2}]}.
    \end{equation}
    Now (\ref{eq:s.1.14}) implies 
    \begin{equation}\label{eq:s.1.20}
        \partial_{\alpha}G = {\partial G \over \partial y_{\alpha}} \in C^{0}(U \times (0, +\infty)) \qquad \text{for $1\leq \alpha \leq n$.}
    \end{equation}
    and {hypothesis 3} implies 
    \begin{equation}\label{eq:s.1.21}
        \partial_{\alpha}G|_{U \times (0,R_{1})} = 0 \qquad \text{for all $1 \leq \alpha \leq n$.}
    \end{equation}
Now (\ref{eq:s.1.20}) allows us to use the continuity of $\partial_{\alpha}G$ over $U \times (0, +\infty)$ to extend the result of (\ref{eq:s.1.21}) up to the "boundary" $U \times \{R_{1}\}$ so that we get 
\begin{equation}\label{eq:s.1.22}
    \partial_{\alpha}G|_{U \times (0,R_1]} \equiv 0 \qquad \text{for $1 \leq \alpha \leq n$}
\end{equation}
On the other hand, (\ref{eq:s.1.19}) implies
\begin{equation}\label{eq:s.1.23}
    \partial_{\alpha} \tilde{G}|_{U \times (R_{1},R_{2})} \equiv \partial_{\alpha} G |_{U \times (R_1,R_2)} \qquad \text{for $1 \leq \alpha \leq n$.}
\end{equation}
For each $1 \leq \alpha \leq n$, the real analyticity of $\partial_{\alpha}\tilde{G}$ on $U \times (R_1(1 - \delta_0), R_2(1 + \delta_0))$, together with the continuity of $\partial_\alpha G$ on $U \times (0, + \infty)$ allows us to extend the validity of (\ref{eq:s.1.23}) up to the "boundary" $U \times \{R_1\} \bigsqcup U \times \{R_2\}$ so that we have 
\begin{equation}\label{eq:s.1.24}
    \partial_\alpha \tilde{G}|_{U \times [R_1,R_2]} \equiv \partial_\alpha G|_{U \times [R_1,R_2]} \qquad \text{for $1\leq \alpha \leq n$.}
\end{equation}
So (\ref{eq:s.1.24}) and (\ref{eq:s.1.22}) together will give
\begin{equation}\label{eq:s.1.25}
    \partial_\alpha \tilde{G}|_{U \times \{R_1\}} \overset{(\ref{eq:s.1.24})}{=}\partial_\alpha G|_{U \times \{R_1\}} \overset{(\ref{eq:s.1.22})}{=} 0 \qquad \text{for $1 \leq \alpha \leq n$.}
\end{equation}

Next, {hypothesis 2} implies
\begin{equation}\label{eq:s.1.26}
    \partial_k\partial_n G \in C^0(U \times (0,+\infty)) \qquad \text{for $1\leq k \leq n$.}
\end{equation}

(\ref{eq:s.1.19}) implies 
\begin{equation}\label{eq:s.1.27}
    \partial_k\partial_n \tilde{G}|_{U \times (R_1,R_2)} = \partial_k\partial_n G|_{U \times (R_1,R_2)} \qquad \text{for $1 \leq k \leq n$.}
\end{equation}

Now, for each $1 \leq k \leq n$, the real analyticity of $\partial_k\partial_n \tilde{G}$ on $U \times (R_1(1-\delta_0),R_2(1+\delta_0))$ and the continuity of $\partial_k\partial_n G$ on $U \times (0, +\infty)$, as given by (\ref{eq:s.1.26}), together allows us to extend the validity of (\ref{eq:s.1.27}) to the boundary $U \times \{R_1\} \bigsqcup U \times \{R_2\}$ so that we get
\begin{equation}\label{eq:s.1.28}
    \partial_k\partial_n \tilde{G}|_{U \times [R_1,R_2]} = \partial_k\partial_n G|_{U \times [R_1,R_2]} \qquad \text{for $1 \leq k \leq n$.}
\end{equation}

Next, {hypothesis 3} implies
\begin{equation}\label{eq:s.1.29}
    \partial_k\partial_n G|_{U \times (0,R_1)} \equiv 0 \qquad \text{for $1 \leq k \leq n$.}
\end{equation}

For each $1 \leq k \leq n$, the continuity of $\partial_k \partial_n G$ over $U \times (0,+\infty)$, as ensured by (\ref{eq:s.1.26}), allows us to extend the validity of (\ref{eq:s.1.29}) to the boundary $U \times \{R_1\}$ so that we get
\begin{equation}\label{eq:s.1.30}
    \partial_k \partial_n G|_{U \times (0,R_1]} \equiv 0 \qquad \text{for $1 \leq k \leq n$.}
\end{equation}

Now (\ref{eq:s.1.28}), together with (\ref{eq:s.1.30}) imply:
\begin{equation}\label{eq:s.1.31}
    \partial_k \partial_n \tilde{G}|_{U \times \{R_1\}} \overset{(\ref{eq:s.1.28})}{=} \partial_k\partial_n G|_{U \times \{R_1\}} \overset{(\ref{eq:s.1.30})}{=} 0 \qquad \text{for all $1 \leq k \leq n$.}
\end{equation}

Now {hypothesis 3} and (\ref{eq:s.1.19}) clearly give 
\begin{equation}\label{eq:s.1.32}
    \tilde{G}|_{U \times \{R_{1}\}} \overset{(\ref{eq:s.1.19})}{\equiv} G|_{U \times \{R_1\}} \overset{\text{hypothesis 3}}{\equiv} 0
\end{equation}
So, by invoking {Lemma \ref{lem:S.2}}, we derive from (\ref{eq:s.1.32}) to get 
\begin{equation}\label{eq:s.1.33}
    \tilde{G}(y) = (y_n - R_1)\cdot \tilde{G}_{1}(y) \qquad \text{for all $y \in U \times (R_1(1-\delta_0),R_{2}(1+\delta_0))$.}
\end{equation}
where
\begin{equation}\label{eq:s.1.34}
    \text{$\tilde{G}_{1}:U \times (R_1(1-\delta_0),R_{2}(1+\delta_0))\to \real^{1}$ is real analytic on $U \times (R_1(1-\delta_0),R_{2}(1+\delta_0))$.}
\end{equation}
(\ref{eq:s.1.23}) gives 
\begin{equation}\label{eq:s.1.35}
    \partial_n \tilde{G}(y) = \tilde{G}_{1}(y) + (y_n -R_1)\cdot\partial_n \tilde{G}_1 (y) \qquad\text{for all $y \in U \times (R_1(1-\delta_0),R_{2}(1+\delta_0))$}
\end{equation}
Now, (\ref{eq:s.1.35}) and (\ref{eq:s.1.25}) together gives 
\begin{equation}\label{eq:s.1.36}
    0 \overset{(\ref{eq:s.1.29})}{\equiv} \partial_n \tilde{G}|_{U \times \{R_1\}} \overset{(\ref{eq:s.1.35})}{=} \tilde{G}_1|_{U \times \{R_1\}}
\end{equation}
Again, by invoking {Lemma \ref{lem:S.2}}, we derive from (\ref{eq:s.1.36}) and (\ref{eq:s.1.34}) to deduce that
\begin{equation}\label{eq:s.1.37}
    \tilde{G}_1(y) = (y_n - R_1)\cdot \tilde{G}_2(y) \qquad\text{for $y \in U \times (R_1(1-\delta_0),R_{2}(1+\delta_0))$,}
\end{equation}
where
\begin{equation}\label{eq:s.1.38}
    \text{$\tilde{G}_2:U \times (R_1(1-\delta_0),R_{2}(1+\delta_0)) \to \real^1$ is real analytic on $U \times (R_1(1-\delta_0),R_{2}(1+\delta_0))$.}
\end{equation}
We substitute (\ref{eq:s.1.37}) back to (\ref{eq:s.1.35}) to get
\begin{equation}\label{eq:s.1.39}
    \partial_n \tilde{G}(y) = 2(y_n - R_1)\cdot \tilde{G}_2(y) + (y_n - R_1)^2 \partial_n\tilde{G}_{2}(y) \qquad \text{for all $y \in U \times (R_1(1-\delta_0),R_{2}(1+\delta_0))$.}
\end{equation}
(\ref{eq:s.1.37}) implies
\begin{equation}\label{eq:s.1.40}
    \partial_n^2\tilde{G}(y) = 2\cdot 1\cdot \tilde{G}_2 (y) + 4(y_n - R_1) \partial_n\tilde{G}_2 (y) + (y_n - R_1)^2 \partial_n^2 \tilde{G}_2(y) \qquad \text{for all $y \in U \times (R_1(1-\delta_0),R_{2}(1+\delta_0))$.}
\end{equation}

Now, (\ref{eq:s.1.31}) and (\ref{eq:s.1.40}) together gives
\begin{equation}\label{eq:s.1.41}
    0 \overset{(\ref{eq:s.1.31})}{=} \partial_n^2 \tilde{G}|_{U \times \{R_1\}} \overset{(\ref{eq:s.1.40})}{=} \tilde{G}_2|_{U \times \{R_1\}}
\end{equation}
By invoking {Lemma \ref{lem:S.2}} again, we derive from (\ref{eq:s.1.38}) and (\ref{eq:s.1.41}) to deduce that
\begin{equation}\label{eq:s.1.42}
    \tilde{G}_2(y) = (y_n- R_1)\cdot \tilde{G}_3(y) \qquad \text{for all $y\in U \times (R_1(1-\delta_0),R_2(1+\delta_0))$,}
\end{equation}
where
\begin{equation}\label{eq:s.1.43}
    \text{$\tilde{G}_3:U \times (R_1(1-\delta_0),R_2(1+\delta_0)) \to \real^1$ is real analytic on $U \times (R_1(1-\delta_0),R_2(1+\delta_0))$.}
\end{equation}
We now simply substitute (\ref{eq:s.1.37}) and (\ref{eq:s.1.42}) successively back to (\ref{eq:s.1.33}) to get 
\begin{equation}\label{eq:s.1.44}
    \begin{aligned}
        \tilde{G}(y) = (y_n -R_1)\cdot \tilde{G}_1(y) &= (y_n - R_1)^2 \tilde{G}_2(y) \\&=(y_n - R_1)^3\tilde{G}_3(y) \qquad \text{for all $y \in U \times (R_1(1-\delta_0),R_{2}(1+\delta_0))$}
    \end{aligned}
\end{equation}

Now, by invoking {hypothesis 3}, (\ref{eq:s.1.19}), (\ref{eq:s.1.44}) and (\ref{eq:s.1.43}), we deduce with basic calculus that $G \in C^{2}(U \times (0,R_2))$ which is precisely (\ref{eq:s.1.18}) as claimed.
\end{proof}

\begin{lem}\label{lemma5.1}
    Let $n \in \mathbb{Z}^{+}$ with $n \geq 2$. Consider $\Psi_{n}:\mathbb{R}^{1} \to \mathbb{R}^{1}$ be the $(n+4)$-degree polynomial given by
    \begin{equation}\label{eq:5.1.1}
        \Psi_{n}(t) = {1\over (n+3)(n+4)}\left(t^{n+4} - 1\right) + {1\over n(n+1)}t^{2}\left(t^{n} - 1\right) - {2\over (n+1)(n+3)}t(t^{n+2} - 1)
    \end{equation}
    Then it follows that
    \begin{equation}\label{eq:5.1.2}
        \Psi_{n}(t) = \sum_{k= 5}^{n+4} {\Psi^{(k)}_{n}(1)\over k!}\cdot(t- 1)^{k} \qquad \text{for all $t \in \mathbb{R}^{1}$,}
    \end{equation}
    and that
    \begin{equation}\label{eq:5.1.3}
        \Psi^{(k)}_{n}(1) > 0 \qquad \text{for all $k \geq 5$.}
    \end{equation}
\end{lem}

\begin{proof}
Starting from \eqref{eq:5.1.1}, a straight forward calculation immediately leads to the following information:
\begin{align}
    \Psi_{n}(1) &= \Psi'_{n}(1) = \Psi''_{n}(1) = \Psi^{(3)}_{n}(1) = \Psi^{(4)}_{n}(1) = 0. \label{eq:5.1.4} \\
    \Psi^{(5)}_{n}(1) &= 2(n+2). \label{eq:5.1.5} \\
    \Psi^{(6)}_{n}(1) &= 6(n+2)(n-1). \label{eq:5.1.6}
\end{align}

(\ref{eq:5.1.4}), (\ref{eq:5.1.5}) and (\ref{eq:5.1.6}) gives
\begin{equation}
    \Psi_{n}(t) = \sum_{k = 5}^{n+4} \frac{\Psi^{(k)}_{n}(1)}{k!}(t-1)^{k}, \qquad \text{for all $t \in \mathbb{R}^{1}$.}
\end{equation}
which is exactly (\ref{eq:5.1.2}) as claimed in {Lemma \ref{lemma5.1}}.

In order to prove (\ref{eq:5.1.3}), we have to view (\ref{eq:5.1.1}) in an alternative way: By using the formula $t^{n} - 1 = (t-1)(t^{n-1} + t^{n-2} + \cdots + t + 1)$ for $t \in \mathbb{R}^{1}$, a direct computation gives:

\begin{equation} \label{eq:5.1.7}
\begin{aligned}
    \Psi_{n}(t) =  (t-1)&\left\{ {1\over(n+3)(n+4)}\sum_{\alpha = 0}^{n+3}t^{\alpha} + {1\over n(n+1)}t^{2}\sum_{\beta}^{n-1}t^{\beta} - {2\over (n+1)(n+3)}t\sum_{\gamma = 0}^{n+1}t^{\gamma} \right\} \\ 
    =(t-1) &\left\{ {1\over (n+3)(n+4)}(1+t) - {2t\over (n+1)(n+3)} + {12\over n(n+1)(n+3)(n+4)}\sum_{\alpha = 2}^{n+1}t^{\alpha} + \right.\\ 
     &\left. + {1\over (n+3)(n+4)}(t^{n+2} + t^{n+3}) - {2\over (n+1)(n+3)}t^{n+2}\right\} \qquad \text{for $t \in \mathbb{R}^{1}$.}
\end{aligned}
\end{equation} 

Now we use the change of variable $\mu = t -1$. So (\ref{eq:5.1.2}) and (\ref{eq:5.1.7}) gives 
\begin{equation}
    \sum_{k = 5}^{n+4} {\Psi^{(4)}_{n}\over k!}\mu^{k} = \mathscr{P}_{1}(\mu) + \mathscr{P}_{2}(\mu) + \mathscr{P}_{3}(\mu), \qquad \text{for $\mu \in \mathbb{R}^{1}$.} \label{eq:5.1.8}
\end{equation}
where
\begin{align}
    \mathscr{P}_{1}(\mu) :&=  \mu \cdot\left\{{2+\mu\over (n+3)(n+4)} - {2(1+\mu)\over (n+1)(n+3)}\right\}\label{eq:5.1.9}\\
    \mathscr{P}_{2}(\mu) :&= {12\mu\over n(n+1)(n+3)(n+4)}\sum_{\alpha = 2}^{n+1} \sum_{a = 0}^{\alpha}\binom{\alpha}{a} \mu^{a}  \label{eq:5.1.10}\\
    \mathscr{P}_{3}(\mu) :&=  \mu\cdot \left\{{1\over (n+3)(n+4)}\left((1+\mu)^{n+2} + (1+\mu)^{n+3}\right) - {2\over (n+1)(n+3)}(1+\mu)^{n+2}\right\} \notag \\
     &= {\mu\over (n+3)(n+4)}\left\{\sum_{\alpha = 0}^{n+2}{\alpha(n+7) - 6(n+3)\over (n+1)(n+3 - \alpha)}\cdot\mu^{\alpha} + {\mu^{n+3}\over (n+4)}\right\}\label{eq:5.1.11}
\end{align}
Observes that 
\begin{equation}
    \text{for all "$6 \leq \alpha \leq n+2$", we have ${\alpha(n+7) - 6(n+3)\over (n+1)(n+3-\alpha)}\geq {24\over (n+1)(n+3 - \alpha)} > 0$} \label{eq:5.1.12}
\end{equation}
The fact that "all coefficients of $\mathscr{P}_{2}$ are positive", together with (\ref{eq:5.1.8}), $\deg\mathscr{P}_{1} = 2$ and (\ref{eq:5.1.12}) immediately imply that 
\begin{equation}
    \Psi^{(k)}_{n}(1)>0 \qquad \text{for all $7 \leq k \leq n+4$.} \label{eq:5.1.13}
\end{equation}

Now, by gathering (\ref{eq:5.1.5}), (\ref{eq:5.1.6}) and (\ref{eq:5.1.13}), we conclude that (\ref{eq:5.1.3}) holds. The proof of {Lemma\ref{lemma5.1}} is completed.
\end{proof}

\begin{lem}\label{lem:s.4}
    Let $n \in \integ^{+}$ with $n \geq 2$, $l = 0,1$ and $0< R_1 < R_2 < +\infty$. Let $A_{1k}:S^{n-1} \to \real^{1}$, $A^{2k}:S^{n-1} \to \real^1$ be prescribed real analytic functions on $S^{n-1}$, for $k = 0,1,2,\dots ,l$. Then it follows that there exists some uniquely determined real-analytic functions $\mathscr{C}_0,\mathscr{C}_{1}, \dots,\mathscr{C}_{2l+1}$ on $S^{n-1}$, each expressible as a finite linear combinations of $A_{1k}$ and $A_{2k}$ (for $0\leq k \leq l$) with constant coefficients such that the real analytic function $Q:\real^n - \{0\} \to \real^1$, which is defined by
    \begin{equation}\label{eq:s.4.1}
        Q(\rho w ) = \sum_{\alpha = 0}^{2l+1}\mathscr{C}_{\alpha}(w) \rho^{\alpha}, \qquad \text{for all $\rho > 0$, $w \in S^{n-1}$}
    \end{equation}
will satisfy the following conclusions:
\begin{align}
    &\left.{d^k\over d\rho^{k}}\right|_{\rho = R_1}Q(\rho w) = A_{1k}(w) \qquad \text{for all $0 \leq k \leq l$, and all $w \in S^{n-1}$}\label{eq:s.4.2}\\
    &\left.{d^k\over d\rho^k}\right|_{\rho = R_2}Q(\rho w) = A_{2k}(w) \qquad \text{for all $0 \leq k \leq l$, and all $w \in S^{n-1}$} \label{eq:s.4.3}\\
    &\int_{A(R_1,R_2)} Q(x) \left.\mathrm{Vol}_{\real^n}\right|_{x} = 0 \label{eq:s.4.4}
\end{align}
\end{lem}

In the case of $l = 0$, the real analytic functions $\mathscr{C}_{0},\mathscr{C}_{1},\mathscr{C}_{2}$ on $S^{n-1}$ satisfying (\ref{eq:s.4.1}), (\ref{eq:s.4.2}), (\ref{eq:s.4.3}) and (\ref{eq:s.4.4}) are given as follows: \begin{align}
    \mathscr{C}_{2} =& {(n+2)\over R_1^2 \sum_{\alpha = 2}^{n+2}\binom{n+2}{\alpha}\mu^{\alpha}}\left\{\left(\sum_{\beta = 1}^{n}\binom{n+1}{\beta +1}\mu^{\beta}\right)A_{10} +\left(\sum_{\beta = 0}^{n}\binom{n+1}{\beta +1} \beta\mu^{\beta}\right)A_{20}\right\} \label{eq:s.4.5}
\end{align}
\begin{equation}
\begin{aligned}
    \mathscr{C}_{1} =&
        - {n+1\over R_1 \sum_{\alpha = 2}^{n+2}\binom{n+2}{\alpha}(\alpha - 2)\mu^{\alpha}}\left( (n+2)n\mu + 2 \sum_{\alpha = 3}^{n+2}\binom{n+2}{\alpha}\mu^{\alpha - 1} \right)A_{10}\\ &- {1\over R_1 \sum_{\alpha =2}^{n+2}\binom{n+2}{\alpha}(\alpha - 2)\mu^{\alpha}}\left(\sum_{\alpha = 1}^{n+2}\binom{n+2}{\alpha}(\alpha - 1)(2n+2-\alpha)\mu^{\alpha - 1}\right)A_{20} \label{eq:s.4.6}
\end{aligned}
\end{equation}
\begin{equation}
\begin{aligned}
    \mathscr{C}_{0} =& {1\over \sum_{\alpha = 2}^{n+2}\binom{n+2}{\alpha}(\alpha - 2)\mu^{\alpha}}\left( n \mu^{n+2}+ \sum_{\alpha = 2}^{n+1}\binom{n+2}{\alpha} {n^2 +2n\over \alpha +1}\mu^{\alpha} + {(n+2)(n+1)n\over 2} \mu \right)A_{10} \\ &+ {1\over \sum_{\alpha = 2}^{n+2}\binom{n+2}{\alpha}(\alpha - 2)\mu^{\alpha}}\left( \sum_{\beta = 0}^{n+1} n\beta \binom{n+2}{\beta +1}\mu^{\beta} \right)A_{20} \label{eq:s.4.7}
\end{aligned}
\end{equation}
where $\mu = {R_{2}\over R_{1}} - 1$.

\begin{proof}
    Here, we always use the abbreviation
\[
t := \frac{R_2}{R_1} > 1
\qquad\mbox{ and }\quad
\mu := t-1 = \frac{R_2}{R_1}-1 > 0.
\]

We just observe that for $l=0$, proving the existence of real-analytic functions
$\mathscr C_0, \mathscr C_1, \mathscr C_2$ on $S^{n-1}$ satisfying
\eqref{eq:s.4.2} to \eqref{eq:s.4.4} is the same as proving that the following linear-algebra problem
\begin{equation}\label{eq:s.5.1}
\begin{bmatrix}
1 & R_1 & R_1^2 \vspace{5pt}\\
1 & R_2 & R_2^2 \vspace{5pt}\\
\dfrac{R_2^n - R_1^n}{n} &
\dfrac{R_2^{n+1}-R_1^{n+1}}{n+1} &
\dfrac{R_2^{n+2}-R_1^{n+2}}{n+2}
\end{bmatrix}
\begin{bmatrix}
\mathscr C_0 \\
\mathscr C_1 \\
\mathscr C_2
\end{bmatrix}
=
\begin{bmatrix}
A_{10} \\
A_{20} \\
0
\end{bmatrix}
\end{equation}
has a unique solution $[\mathscr C_0, \mathscr C_1, \mathscr C_2]^\top$.

By elementary row-reduction, system \eqref{eq:s.5.1} is equivalent to

\begin{equation}\label{eq:s.5.2}
\begin{aligned}
&\begin{bmatrix}
1 & R_1 & R_1^2 \vspace{5pt}\\
0 & 1 & R_1+R_2 \vspace{5pt}\\
0 & 0 &
-\dfrac{R_1^{n+2}}{n(n+1)(n+2)} \,
\widetilde{\Psi}_n(t)
\end{bmatrix}
\begin{bmatrix}
\mathscr C_0 \\
\mathscr C_1 \\
\mathscr C_2
\end{bmatrix}\\
&\qquad\qquad=
\begin{bmatrix}
A_{10} \vspace{5pt}\\
\dfrac{1}{R_2-R_1}\cdot\Big(A_{20}-A_{10}\Big) \vspace{5pt}\\
-\dfrac{R_2^n-R_1^n}{n} A_{10}
-\dfrac{R_1^{n+1}}{n(n+1)}
\Big( nt^{n+1} - (n+1)t^n + 1 \Big)
\dfrac{A_{20}-A_{10}}{R_2-R_1}
\end{bmatrix},
\end{aligned}
\end{equation}

where

\begin{equation}\label{eq:s.5.3}
\widetilde{\Psi}_n(t)
=
n t^{n+2}
-(n+2)t^{n+1}
+(n+2)t
-n
=
\sum_{\alpha=2}^{n+2}
\binom{n+2}{\alpha}
(\alpha-2)
(t-1)^\alpha.
\end{equation}

System \eqref{eq:s.5.3} clearly implies

\begin{equation}\label{eq:s.5.4}
\widetilde{\Psi}_n(t) > 0
\qquad
\text{for all }\quad
t=\frac{R_2}{R_1}>1.
\end{equation}

Hence \eqref{eq:s.5.4} implies that \eqref{eq:s.5.2} is uniquely solvable with the following unique solution set with $t=\frac{R_2}{R_1}>1$.

\begin{equation}\label{eq:s.5.5}
\begin{cases}
\displaystyle
\mathscr C_2
=
\dfrac{n+2}{R_1^2\,\widetilde{\Psi}_n(t)\,(t-1)}\cdot
\Big\{
\left( t^{n+1}-(n+1)t+n \right) A_{10}
+
\left( nt^{n+1}-(n+1)t^n+1 \right) A_{20}
\Big\}, \vspace{5pt}
\\
\displaystyle
\mathscr C_1
=
-(R_1+R_2)\mathscr C_2
+
\dfrac{1}{R_2-R_1}\cdot (A_{20}-A_{10}), \vspace{5pt}
\\
\displaystyle
\mathscr C_0
=
-R_1 \mathscr C_1
-R_1^2 \mathscr C_2
+A_{10}.
\end{cases}
\end{equation}

Now, recall that $t=\frac{R_2}{R_1}>1$ and $t=1+\mu$ with $\mu=\frac{R_2}{R_1}-1$, by using the formula
$$
t^m=(1+\mu)^m=1+\sum_{\alpha=1}^{m}\binom{m}{\alpha}\mu^\alpha
$$
in a repeated and successive manner, a straight-forward computation based on \eqref{eq:s.5.5} will directly lead to conclusions \eqref{eq:s.4.5}, \eqref{eq:s.4.6}, and \eqref{eq:s.4.7}. Since such a long and tedious but routine computation is "brute-force" in nature, and thus un-enlightening, we omit the details and leave it to the reader.

Next, for the case of $l=1$, proving the existence of real-analytic functions
$\mathscr C_0, \mathscr{C}_1,\dots,\mathscr C_4$ on $S^{n-1}$ satisfying
\eqref{eq:s.4.2}, \eqref{eq:s.4.3}, \eqref{eq:s.4.4}
is equivalent to proving that the following linear system

\begin{equation}\label{eq:s.5.6}
M
\begin{bmatrix}
\mathscr C_0\\
\mathscr C_1\\
\mathscr C_2\\
\mathscr C_3\\
\mathscr C_4
\end{bmatrix}
=
\begin{bmatrix}
A_{10}\\
A_{20}\\
A_{11}\\
A_{21}\\
0
\end{bmatrix}
\end{equation}

admits a unique solution $[\mathscr C_0, \mathscr C_1, \dots, \mathscr C_4]^\top$, where

\[
M=
\begin{bmatrix}
1 & R_1 & R_1^2 & R_1^3 & R_1^4 \vspace{5pt}\\
1 & R_2 & R_2^2 & R_2^3 & R_2^4 \vspace{5pt}\\
0 & 1 & 2R_1 & 3R_1^2 & 4R_1^3 \vspace{5pt}\\
0 & 1 & 2R_2 & 3R_2^2 & 4R_2^3 \vspace{5pt}\\
\dfrac{R_2^n-R_1^n}{n} &
\dfrac{R_2^{n+1}-R_1^{n+1}}{n+1} &
\dfrac{R_2^{n+2}-R_1^{n+2}}{n+2} &
\dfrac{R_2^{n+3}-R_1^{n+3}}{n+3} &
\dfrac{R_2^{n+4}-R_1^{n+4}}{n+4}
\end{bmatrix}.
\]

By applying row-reduction technique several times, \eqref{eq:s.5.6} can be reduced to the following equivalent system of equations

\begin{equation}\label{eq:s.5.7}
\begin{bmatrix}
1 & R_1 & R_1^2 & R_1^3 & R_1^4 \vspace{5pt}\\
0 & 1 & 2R_1 & 3R_1^2 & 4R_1^3 \vspace{5pt}\\
0 & 0 & 1 & R_2+2R_1 & R_2^2+2R_2R_1+3R_1^2 \vspace{5pt}\\
0 & 0 & 0 & 1 & 2(R_2+R_1) \vspace{5pt}\\
0 & 0 & 0 & 0 &
\dfrac{2 R_1^{n+4}}{n+2}\cdot\Psi_n(t)
\end{bmatrix}
\begin{bmatrix}
\mathscr C_0\\
\mathscr C_1\\
\mathscr C_2\\
\mathscr C_3\\
\mathscr C_4
\end{bmatrix}
=
\begin{bmatrix}
A_{10}\vspace{5pt}\\
A_{11}\vspace{5pt}\\
\frac{1}{R_2-R_1}\Big(\frac{A_{20}-A{10}}{R_2-R_1}-A_{11}\Big)\vspace{5pt}\\
\frac{A_{21}-A{11}}{(R_2-R_1)^2}-\frac{2}{(R_2-R_1)^2}\Big(\frac{A_{20}-A{10}}{R_2-R_1}-A_{11}\Big)\vspace{5pt}\\
B
\end{bmatrix},
\end{equation}

where

\begin{align*}
\Psi_n(t)
&=
\frac{1}{(n+3)(n+4)}(t^{n+4}-1)
-\frac{2}{(n+1)(n+3)}(t^{n+3}-t)\\
&\quad +\frac{1}{n(n+1)} t^2 (t^n-1)\quad\mbox{ for all }\quad t=\frac{R_2}{R_1}>1
\label{eq:s.5.8}
\end{align*}
and
\begin{equation}
\begin{split}
B=&-\frac{(R_2^n-R_1^n)}{n}\cdot A_{10}-\frac{R_1^{n+1}}{n(n+1)}\cdot(n\cdot t^{n+1}-(n+1)\cdot t^n+1)\cdot A_{11}\\
&\quad-\frac{1}{R_2-R_1}\cdot\Big(\frac{A_{20}-A_{10}}{R_2-R_1}-A_{11}\Big)\cdot\frac{R^{n+2}}{n(n+1)(n+2)}\cdot\\
&\qquad\qquad\Big\{(n+1)\cdot n\cdot t^{n+2}-2n(n+2)\cdot t^{n+1}+(n+1)(n+2)\cdot t^n-2\Big\}\\
&\quad\quad-\frac{R_1^{n+3}}{n(n+1)}\cdot\Big(-\frac{n(n+1)}{(n+3)(n+2)}\cdot t^{n+3}+\frac{2n}{n+2}\cdot t^{n+2}-t^{n+1}+\frac{2t}{n+2}-\frac{2n}{(n+2)(n+3)}\Big)\cdot\\
&\quad\qquad\qquad\Big\{\frac{A_{21}-A_{11}}{(R_2-R_1)^2}-\frac{2}{(R_2-R_1)^2}\cdot(\frac{A_{20}-A_{10}}{R_2-R_1}-A_{11})\Big\}
\label{eq:s.5.9}
\end{split}
\end{equation}

Now, we invoke Lemma s.1 to deduce that we have
\begin{equation}\label{eq:s.5.10}
\Psi_n(t) > 0
\quad
\text{holds for all }\quad t=\frac{R_2}{R_1}>1.
\end{equation}

Hence \eqref{eq:s.5.4} implies that the system \eqref{eq:s.5.7} has a unique solution $[\mathscr C_0, \mathscr C_1, \dots, \mathscr C_4]^\top$. Thus, it follows that system \eqref{eq:s.5.6} also has a unique solution $[\mathscr C_0, \mathscr C_1, \dots, \mathscr C_4]^\top$. The proof of Lemma \ref{lem:s.4} for the case of $l=1$ is completed.
\end{proof}

\newpage
\section{The First Major Step}
\begin{lem}\label{keylemma}
    Consider $n \in \mathbb{Z}^{+}$ with $n \geq 2$ and $l \in \mathbb{Z}^{+}$ with $l = 0,1$. Consider a real-valued function $f \in C^{l}_{c}(\mathbb{R}^{n})$ which satisfies the following properties
    \begin{description}
        \item[hypothesis 1] $\mathrm{supp}\, f \subset \overline{A(R_{1},R_{2})}$, where $A(R_{1},R_{2}) = \{ x \in \mathbb{R}^{n}:R_{1} < |x| < R_{2}\}.$
        \item[hypothesis 2] $\int_{\overline{A(R_{1},R_{2})}}f(x) \, \mathrm{Vol}_{\mathbb{R}^{n}}(x) = 0$.
        \item[hypothesis 3] $f|_{\overline{A(R_{1},R_{2})}}$ is real analytic on $\overline{A(R_{1},R_{2})}$.
    \end{description}
    Define the real-valued function $g_{2}: S^{n-1} \to \mathbb{R}$ on the standard unit sphere by 
    \begin{equation}
        g_{2}(y) = \int_{0}^{R_{2}} f(\rho y)\rho^{n-1} \ \dd\rho = \int_{R_{1}}^{R_{2}}f(\rho y)\rho^{n-1} \ \dd\rho,
    \end{equation}
    for all $y \in S^{n-1}$. Next, consider $\varphi:S^{n-1} \to \mathbb{R}$ be the unique solution to the Poisson equation $- \Delta_{S^{n-1}}\varphi = g_{2}$ on $S^{n-1}$ which satisfies $\int_{S^{n-1}}\varphi \ \mathrm{Vol}_{S^{n-1}} = 0$. Define a $(n-1)$-form $\Psi_{2}$ on $\mathbb{R}^{n}$ as follows: for all $x \in \real^{n}$,
    \begin{equation}
        \left.\Psi_{2}\right|_{x} = \left(\int_{0}^{|x|}f(\rho{x\over |x|})\rho^{n-1} \ \dd\rho\right) {1\over |x|^{n}} \sum_{k = 1}^{n}(-1)^{k}x^{k}\dd x^{1} \wedge \cdots\wedge \dd x^{k-1} \wedge \dd x^{k+1} \wedge \cdots \wedge \dd x^{n} - \left.\dd_{\real^{n}}(\chi\mathbf{r}^{\ast}(\dd_{S^{n-1}}(\varphi \mathrm{Vol_{S^{n-1}}})))\right|_{x}.
    \end{equation} 
    with $\mathbf{r}: \mathbb{R}^{n} - \{O\} \to S^{n-1}$ be the standard retraction map $\mathbf{r}(x) = {x\over |x|}$, and the cut-off function $\chi:\mathbb{R}^{n} \to \mathbb{R}$ is given by
    \begin{equation}
    \chi(x) = \begin{cases}\sin^{l+3}\left(\frac{1}{2}\cdot \frac{R_{1}}{R_{2}- R_{1}}\cdot(|x|- R_{1})\right), & \text{if $R_{1}\leq |x| \leq R_{2}$} \\
    0,& \text{if $x \in \real^n$ satisfies $0 \leq |x| \leq R_{1}$}\\
    1, &\text{if $x \in \real^n$ satisfies $|x| \geq R_{2}$}
    \end{cases} 
\end{equation}
Then it follows that the $(n-1)$-form $\Psi_{2}$ satisfies the following properties:
\begin{itemize}
    \item $\Psi_{2}\in C^{l+1}_{c}(\bigwedge^{n-1}T^{\ast}\mathbb{R}^{n})$.
    \item $\mathrm{supp}\, \Psi_{2} \subset \overline{A(R_{1},R_{2})}$.
    \item $\left.\Psi_{2}\right|_{\overline{A(R_{1},R_{2})}}$ is real analytic on $\overline{A(R_{1},R_{2})}$.
    \item $f\ \mathrm{Vol}_{\mathbb{R}^{n}} = \dd_{\mathbb{R}^{n}}\Psi_{2}$ holds on $\mathbb{R}^{n}$.
\end{itemize}
\end{lem}

\begin{proof}
Let $R_{1} < R_{2}$, $n \in \mathbb{Z}^{+}$ satisfies $n\geq 3$. Take $l \in \mathbb{Z}^{+}$ with $l =0,1$. As in the {hypothesis} of {Lemma} \ref{keylemma}, consider $f \in C^{l}_{c}(\mathbb{R}^{n})$ which satisfies the following properties: 
\begin{description}
    \item[hypothesis 1] \label{eq:1} $\supp f \subset \overline{A(R_{1},R_{2})}$, where $A(R_{1},R_{2}) = \{x \in \mathbb{R}^{n}: R_{1} < |x| < R_{2}\}$.
    \item[hypothesis 2] \label{eq:2} $\int\limits_{A(R_{1},R_{2})}f(x) \, \left.\mathrm{Vol_{\mathbb{R}^{n}}}\right|_{x} = 0$.
    \item[hypothesis 3] \label{eq:3} $\left.f\right|_{\overline{A(R_{1},R_{2})}}$ is real analytic on $A(R_{1},R_{2})$.
\end{description}
Note that {hypothesis 1} is equivalent to
\begin{equation}\begin{aligned}
&\text{There exists some $0 < \delta_{0} <1$ sufficiently small for which we can find}\\ &\text{some real analytic function $\tilde{f}:A(R_{1}(1-\delta_{0}),R_{2}(1+ \delta_{0})) \to \mathbb{R}^{1}$ such that $\tilde{f}|_{A(R_{1},R_{2})} = f$.} \label{eq:4}
\end{aligned}\end{equation}

\textbf{Part 1. Define the differential form $\eta$.}
Consider $F:[0,1]\times \mathbb{R}^{n} \to \mathbb{R}^{n}$ which is given by
\begin{equation}
    F(t,x) = tx \qquad \text{for all $t \in [0,1]$, and all $x \in \real^{n}$.} \label{eq: 5}
\end{equation}
We now strictly follow the idea of Spivak on page 265 of his textbook \cite{Spivak} (see Step(2) of Spivak's proof of Theorem 9 of chapter 8), and we consider the following differential form:
\begin{equation}\label{eq:6}
\begin{aligned}
    \left.\eta\right|_{x} :&= \left(\int_{0}^{|x|} f(\rho \cdot \frac{x}{|x|}) \ \dd\rho\right) \frac{1}{|x|^{n}}\cdot \sum_{k = 1}^{n}(-1)^{k+1}x^{k} \dd x^{1} \wedge \cdots \wedge \widehat{\dd x^{k}} \wedge \cdots \wedge \dd x^{n} \\
    & = \int_{0}^{1} f(\tau x)\tau^{n-1}d\tau \cdot \sum_{k = 1}^{n} (-1)^{k+1}x^{k} \dd x^{1} \wedge \cdots \wedge \widehat{\dd x^{k}} \wedge \cdots \wedge \dd x^{n} \quad \text{for all $x \in \real^{n}$}. 
\end{aligned}
\end{equation}

\textbf{Part 2. Regularity of $\eta$.}
Showing that $\eta \in C^{l+1}(\bigwedge^{n-1}T^{\ast}\mathbb{R}^{n})$ amounts to the same as showing that the following function belongs to $C^{l+1}(\mathbb{R}^{n})$
\begin{equation}\label{eq:add.1}
    \mathscr{H}(x) = \int_{0}^{|x|}f(\rho{x\over |x|})\rho^{n-1} \ \dd\rho \qquad \text{for all $x \in \mathbb{R}^{n}$.}    
\end{equation}
That is, we need to show 
\begin{equation} \label{eq:add.2}
\mathscr{H} \in C^{l+1}(\real^{n}).
\end{equation}
Note that $\mathscr{H}|_{B_{0}(R_{1})} \equiv 0$, so as long as we can prove
\begin{equation} \label{eq:add.3}
    \text{$\mathscr{H}\in C^{l+1}(\mathbb{R}^{n} - \{tv:t > 0\})$, fo any unit vector $v \in S^{n-1}$}
\end{equation}
It follows that \eqref{eq:add.2} holds.

Now, take $v \in S^{n-1}$, up to a rotation about the origin, we may treat the line $\{tv: t >0\}$ as the positive $x^{n}$-axis of some suitable Euclidean coordinate system $(x^{1},\dots,x^{n})$. Next, choose the spherical coordinate system $(\rho,\theta^{1},\dots,\theta^{n-1})$ on $\real^n - \{tv:t > 0\}$, adopted to that Euclidean coordinate system, and we consider the associated local parametrisation $\Phi = (\rho,\theta^{1},\dots, \theta^{n-1}):(0 , +\infty) \times (0,\pi)^{n-2} \times (0,2\pi) \to \real^{n} - \{tv:v>0\}$. Take $G:(0 , +\infty) \times (0,\pi)^{n-2} \times (0,2\pi) \to \real^1$ be given by $G = \mathscr{H} \circ \Phi$. It is clear that proving \eqref{eq:add.3} amounts to the same as showing that 
\begin{equation}\label{eq:add.4}
    G \in C^{l+1}((0, +\infty) \times (0,\pi)^{n-2} \times (0,2\pi))
\end{equation}
For convenience, we write $U = (0,\pi)^{n-2} \times (0,2\pi) \subset \mathbb{R}^{n-1}$.

Observe that $f \in C^{l}_{c}(\real^n)$ and {hypothesis 1}, {hypothesis 3} and \eqref{eq:add.3} implies that $G$ satisfies property 1 and property 2 of {Lemma \ref{lem:s.3}} and the hypothetical condition of statement 1 of {Lemma \ref{lem:s.3}}, we deduce that \begin{equation}\label{eq:add.5}G \in C^{l+1}((0,R_{2})\times U),\end{equation} with $U = (0,\pi)^{n-2} \times (0,2\pi) \subset \mathbb{R}^{n-1}$. (Note that here the $\rho$-direction plays the role of the $y_{n}$-direction on {Lemma \ref{lem:s.3}}).

Next, in order to show that
\begin{equation}\label{eq:add.6}
    G \in C^{l+1}((R_{1}, +\infty) \times U),
\end{equation}
consider $\mathscr{G}:U \to \real^{1}$ given by
\begin{equation}\label{eq:add.7}
\begin{aligned}
    \mathscr{G}(\theta^1, \dots, \theta^{n-1}) &= \int_{R_{1}}^{R_{2}} f\circ \Phi(\rho,\theta^{1},\dots,\theta^{n-1})\rho^{n-1} \ \dd\rho \\
    &= \int_{R_{1}}^{R_{2}}f(\rho\cdot \Phi(1, \theta^{1}, \dots, \theta^{n-1}))\rho^{n-1} \ \dd\rho \qquad \text{for $(\theta^{1},\dots, \theta^{n-1}) \in U$.}
\end{aligned}
\end{equation}

{hypothesis 3} and \eqref{eq:add.7} ensures that
\begin{equation}\label{eq:add.8}
    \text{$\mathscr{G}$ is real-analytic on $U$.}
\end{equation}
Consider the projection $\pi_{U}:(0, +\infty)\times U \to U$ given by $\pi_{U}(\rho, \theta^{1}, \dots,\theta^{n-1}) = (\theta^{1}, \dots,\theta^{n-1})$. Define $G_{\sharp}:U \times(0,+\infty) \to \real$ by $G_{\sharp} = G - \mathscr{G} \circ \pi_{U}$. By our construction and \eqref{eq:add.8}, etc., it is easy to check that $G_{\sharp}$ satisfies {property 1}, {property 2} of {Lemma \ref{lem:s.3}}. So by invoking {statement 2} of {Lemma \ref{lem:s.3}}, we deduce that $$G - \mathscr{G} \circ \pi_{U} \in C^{l+1}((R_{1}, +\infty) \times U).$$Thus $G \in C^{l+1}((R_1, + \infty) \times U)$ must hold. This gives the validity of \eqref{eq:add.6}. Now, \eqref{eq:add.5} and \eqref{eq:add.6} imply \eqref{eq:add.4}.
By tracing back all the logical steps, the validity of \eqref{eq:add.4} will imply the validity of \eqref{eq:add.1}, and thus the validity of \eqref{eq:15}.

\textbf{Part 3. About Poisson equation and definition of $\Psi_{2}$.}
Now, \eqref{eq:6}, {hypothesis 1}, {hypothesis 3} together implies 
\begin{equation}
    \eta \in C^{l+1}(\bigwedge^{n-1}T^{\ast}\real^{n}) \label{eq:15}
\end{equation}
It is well-known that
\begin{equation}
    f\cdot \mathrm{Vol}_{\mathbb{R}^{n}} = d\eta \qquad \text{holds on $\real^{n}$} \label{eq:7}    
\end{equation}
(this is exactly the first line which appears in page 367 of \cite{Spivak}.)

Defines
\begin{equation}
    g(y) = \int_{R_{1}}^{R_{2}} f(\rho y) \rho^{n-1} \ \dd\rho = \int_{0}
^{R_{2}} f(\rho y)\rho^{n-1} \ \dd\rho \qquad \text{for all $y \in S^{n-1}$} \label{eq:8}
\end{equation}

Note that the second equal sign of \eqref{eq:8} comes from {hypothesis 1}. Then {hypothesis 3} implies that $g$ is real-analytic on $S^{n-1}$.

Now {hypothesis 1} implies that $\eta$ as given by \eqref{eq:6} satisfies: 
\begin{equation}
    \left.\eta\right|_{x} \equiv 0 \qquad \text{for all $x \in B_{0}(R_{1}) = \{y\in \real^n: |y| < R_{1}\}$.}
\end{equation}
Let $\varphi: S^{n-1} \to \real^{1}$ be the unique solution to the following Poisson equation:
\begin{equation}\label{eq:11}
    \begin{cases}
        - \Delta_{S^{n-1}}\varphi = g & \text{on $S^{n-1}$,} \\ \int_{S^{n-1}}\varphi \mathrm{Vol}_{S^{n-1}} = 0
    \end{cases}
\end{equation}
Since $g$ is real analytic on $S^{n-1}$, standard regularity theory ensures that the solution function $\varphi$ to \eqref{eq:11} must be real-analytic on $S^{n-1}$.

Define
\begin{equation}\label{eq:rmkuse2}\beta:= d^{\ast}_{S^{n-1}}(\varphi \mathrm{Vol}_{S^{n-1}}).\end{equation} Then the real-analyticity of $\varphi$ on the real-analytic submanifold $S^{n-1}$ of $\real^{n}$ implies that $\beta$ is a real-analytic $(n-2)$-form on $S^{n-1}$. A direct computation easily gives
\begin{equation}
    g\mathrm{Vol}_{S^{n-1}} = \dd_{S^{n-1}}\beta \qquad \text{on $S^{n-1}$} \label{eq:12}
\end{equation}
Now \eqref{eq:6}, \eqref{eq:8}, \eqref{eq:12} gives 
\begin{equation}\label{eq:remarkuse1}
\begin{aligned}
    \left.\eta\right|_{x} &= \int_{R_{1}}^{R_{2}}f(\rho\frac{x}{|x|})\rho^{n-1} \ \dd\rho \cdot \left.\mathbf{r}^{\ast}(\mathrm{Vol}_{S^{n-1}})\right|_{x} \\
    &= \left.\mathbf{r}^{\ast}\left(g\mathrm{Vol}_{S^{n-1}}\right)\right|_{x} \\
    &= \left.\mathbf{r}^{\ast}\left(\dd_{S^{n-1}}\beta\right)\right|_{x} \\
    &= \left.\mathbf{r}^{\ast}\left(\dd_{S^{n-1}}\dd_{S^{n-1}}^{\ast}\left(\varphi \mathrm{Vol}_{S^{n-1}}\right)\right)\right|_{x}
\end{aligned}
\end{equation}

where $\mathbf{r}:\real^{n} - \{O\} \to S^{n-1}$ is the standard retraction map given by $\mathbf{r}(x) = \frac{x}{|x|}$ for all $x \in \real^{n} - \{0\}$.

Define the radially symmetric cut-off function $\chi:\real^{n} \to \real^{1}$ by 
\begin{equation}
    \chi(x) = \begin{cases}\sin^{l+3}\left(\frac{1}{2}\cdot \frac{\pi}{R_{2}- R_{1}}\cdot(|x|- R_{1})\right), & \text{if $R_{1}\leq |x| \leq R_{2}$} \\
    0,& \text{if $x \in \real^n$ satisfies $0 \leq |x| \leq R_{1}$}\\
    1, &\text{if $x \in \real^n$ satisfies $|x| \geq R_{2}$}
    \end{cases} \label{eq:13}
\end{equation}

We now consider the following differential form: 
\begin{equation}\label{eq:14}
    \left.\Psi_{2}\right|_{x} = \left.\eta\right|_{x} - \left.\dd_{\real^{n}}(\chi\mathbf{r}^{\ast}(\dd^{\ast}_{S^{n-1}}(\varphi \mathrm{Vol}_{S^{n-1}})))\right|_{x} \qquad \text{for all $x \in \real^{n}$}. 
\end{equation}
where $\eta$ is given by \eqref{eq:6}, and the cut-off function $\chi$ is given by \eqref{eq:13} and $\varphi$ is the real-analytic solution to \eqref{eq:11}.

\textbf{Part 4. Checking the properties and the conclusion.}
Recall that \eqref{eq:6}, {hypothesis 1}, {hypothesis 3} together implies 
\begin{equation}
    \eta \in C^{l+1}(\bigwedge^{n-1}T^{\ast}\real^{n})
\end{equation}
And \eqref{eq:13} implies 
\begin{equation}
    \chi \in C^{l+2}(\real^{n}) \label{eq:16}
\end{equation}
Thus, \eqref{eq:16} and \eqref{eq:15} as applied to \eqref{eq:14} implies
\begin{equation}
    \Psi_{2} \in C^{l+1}(\bigwedge^{n-1}T^{\ast}\real^{n}) \label{eq:17}
\end{equation}
(note that in getting \eqref{eq:17}, we have also used the real-analyticity of $\mathbf{r}^{\ast}\left(\dd^{\ast}_{S^{n-1}}(\varphi \mathrm{Vol}_{S^{n-1}})\right)$ on $\real^{n} - \{0\}$.)

Now \eqref{eq:12} together with $\chi \equiv 1$ on $\real^{n} - B_{0}(R_{2})$ imply
\begin{equation}\begin{aligned}
    \left.\Psi_{2}\right|_{x}  &\overset{\text{\eqref{eq:14}}}{=}\left.\eta\right|_{x} - \left.\dd_{\real^{n}}\left(1\cdot \mathbf{r}^{\ast}\left(\varphi \mathrm{Vol}_{S^{n-1}}\right)\right)\right|_{x} \\
    &= \left.\mathbf{r}^{\ast}\left( \dd_{S^{n-1}}\dd^{\ast}_{S^{n-1}}\left(\varphi \mathrm{Vol}_{S^{n-1}}\right)\right)\right|_{x} - \left. \dd_{\real^{n}}\left(\mathbf{r}^{\ast}\left(\dd^{\ast}_{S^{n-1}}\left(\varphi \mathrm{Vol}_{S^{n-1}}\right)\right)\right)\right|_{x} \\
    &= 0 \qquad \text{for all $x \in \mathbb{R}^{n} - B_{0}(R_{2})$.} \label{eq:18}
\end{aligned}\end{equation}
Also note that {hypothesis 1} and \eqref{eq:6} implies
\begin{equation}
    \left. \eta \right|_{x} \equiv 0 \qquad \text{for all $x \in \overline{B_{0}(R_{1})}$.} \label{eq:19}
\end{equation}
\eqref{eq:19} together with $\chi \equiv 0$ on $\overline{B_{0}(R_{1})}$ imply 
\begin{equation}
    \left. \Psi_{2} \right|_{x} = 0 - \left. \dd_{\mathbb{R}^{n}}\left(0\cdot \mathbf{r}^{\ast}\left (\dd^{\ast}_{S^{n-1}}\left(\varphi \mathrm{Vol}_{S^{n-1}}\right)\right)\right) \right|_{x} = 0 \qquad {\text{for all $x \in \overline{B_{0}(R_{1})}$}.} \label{eq:20}
\end{equation}
Now \eqref{eq:18} and \eqref{eq:20} imply: 
\begin{equation}
\mathrm{supp}\, \Psi_{2} \subset \overline{A(R_{1},R_{2})} \label{eq:21}
\end{equation}

Note that \eqref{eq:6} and {hypothesis 1} implies 
\begin{equation}
    \left. \eta \right|_{x} = \int_{R_{1}}^{|x|} f(\rho \cdot \frac{x}{|x|})\rho^{n-1} \ \dd\rho \cdot \left.\mathbf{r}^{\ast}(\mathrm{Vol}_{S^{n-1}})\right|_{x} \qquad \text{for all $x \in \overline{A(R_{1},R_{2})}$} \label{eq:22}
\end{equation}

Then \eqref{eq:22} and {hypothesis 3} imply that
\begin{equation}
    \left.\eta \right|_{\overline{A(R_{1},R_{2})}} \text{ is real analytic on $\overline{A(R_{1},R_{2})}$} \label{eq:23} 
\end{equation}

Note \eqref{eq:13} and \eqref{eq:14} imply
\begin{equation}
    \left. \Psi_{2} \right|_{x} = \left. \eta \right|_{x} - \left.\dd_{\mathbb{R}^{n}}\left( \sin^{l+3}\left( \frac{\pi}{2(R_{2} - R_{1})}(|x| - R_{1})\right)\cdot \mathbf{r}^{\ast}(\dd_{S^{n-1}}(\varphi \mathrm{Vol}_{S^{n-1}})) \right)\right|_{x} \qquad \text{for all $x \in \overline{A(R_{1},R_{2})}$} \label{eq:24}
\end{equation}

So \eqref{eq:23} and the real-analyticity of $\varphi$ on $S^{n-1}$ imply
\begin{equation}
    \text{$\left.\Psi_{2}\right|_{\overline{A(R_{1},R_{2})}}$  is real analytic on $\overline{A(R_{1},R_{2})}$.} \label{eq:25}
\end{equation}

Of course, \eqref{eq:7} and \eqref{eq:14} gives 
\begin{equation}
f \cdot \mathrm{Vol}_{\mathbb{R}^n} = \dd_{\mathbb{R}^{n}}\Psi_{2} \qquad \text{on $\mathbb{R}^{n}$} \label{eq:26}   
\end{equation}

By gathering results: \eqref{eq:17}, \eqref{eq:21}, \eqref{eq:25} and \eqref{eq:26}. The proof of {Lemma \ref{keylemma}} is completed.
\end{proof}

\begin{remark}
For the case $n = 2$, all the steps are the same except the analyticity of $\beta$ define in \eqref{eq:rmkuse2}. Here is the argument. The equation \eqref{eq:8} defines a function on $S^{1}$. By choosing a parametrization $\theta \in [0,2\pi)$ of $S^{1}$, the form $g(y)\mathrm{Vol}_{S^{n-1}}$ is essentially $g(\theta) \ \dd\theta$. The {hypothesis 3} implies 
\begin{equation}\label{eq:rmk1}
    \int_{0}^{2\pi} g(\theta) \ \dd\theta = 0
\end{equation}
Recall that the de Rham theorem says that $$\int_{S^1}:H^{1}(S^{1}) \to \real$$ is an isomorphism. So \eqref{eq:rmk1} implies that $g \ d\theta$ is an exact form. Suppose $g \ d\theta = d\beta$ for some smooth function $\beta(\theta)$ on $S^{1}$. In fact $\beta(\theta)$ can be explicitly calculated as
\begin{equation}\label{eq:rmk2}
    \beta(\theta) = \int_{0}^{\theta} g(\tau) \ d\tau + \mathrm{const.}
\end{equation}
Since $g$ is a real-analytic function, the function $\beta$ is real analytic also. Hence Lemma \ref{keylemma} is also valid for $n = 2$.
\end{remark}

We need to see the way in which Lemma \ref{keylemma} directly leads to a real analytic solution to the divergence equation $\dv_{\mathbb{R}^n} v = f$ with the function $f$ satisfying the same hypothesis as stated in Lemma \ref{keylemma}. We treat $\mathbb{R}^n$ as a Riemannian manifold with the standard Euclidean metric $g_{\mathbb{R}^n}(\cdot , \cdot )$, and apply the Hodge star operator
\begin{equation*}
*_{\mathbb{R}^n} : \bigwedge^{n-1}T_x^*\mathbb{R}^n \rightarrow T_x^*\mathbb{R}^n.
\end{equation*}

In what which follows, we will also employ the natural isomorphism $\theta \mapsto \theta_{\sharp}$ from  $T_x^*\mathbb{R}^n$ to $T_x\mathbb{R}^n$ as given by

\begin{equation*}
\theta = g_{\mathbb{R}^n} ( \theta_{\sharp} , \cdot ) .
\end{equation*}
With the above preparation, it is plain to see that Lemma \ref{keylemma} immediately gives the following corollary.




\begin{cor}\label{corollaryfromCohomology}
Take $l =0,1$ and $n \in \mathbb{Z}^+$ with $n \geq 2$.  Consider two positive numbers $0< R_1 < R_2 < \infty$ and a real-valued function $f \in C_c^{l} \big ( \mathbb{R}^n \big )$ which satisfies the following properties
\begin{itemize}
\item $\supp f \subset \overline{A( R_1, R_2)}$ , where $A( R_1, R_2 ) =\{ x \in \mathbb{R}^n : R_1 < |x| < R_2 \}$.
\item $\int_{\overline{A( R_1, R_2)}} f(x) \Vol_{\mathbb{R}^n}(x) = 0$.
\item $f \Big |_{\overline{A( R_1, R_2 )}}$ is real analytic on $\overline{A( R_1, R_2 )}$.
\end{itemize}
For the case of $n \geq 3$, we consider the vector field $v$ on $\mathbb{R}^n$ as given by
\begin{equation}\label{verygoodvectorfield}
v \big |_x = (-1)^{3n+1} \Big ( *_{\mathbb{R}^n} \Big \{ \int_{0}^1 f(\tau x ) \tau^{n-1} \dd \tau \textbf{r}^*(\Vol_{S^{n-1}} ) \big |_x
- \dd_{\mathbb{R}^n} \Big ( \chi \textbf{r}^* \big ( \dd_{S^{n-1}}^* (\varphi \Vol_{S^{n-1}}) \big )  \Big ) \big |_x  \Big \} \Big )_{\sharp} ,
\end{equation}
where the cut off function is given by
\begin{equation}\label{chiGoodGood}
\chi ( x ) =  \sin^{l+3} \Big ( \frac{\pi}{2 ( R_2 - R_1  )} (|x| - R_1 ) \Big ) \mathbf{1}_{[R_1, R_2]} (|x|) + \mathbf{1}_{[ R_2, \infty  )} (|x|) ,
\end{equation}
and that $\varphi : S^{n-1} \rightarrow \mathbb{R}$ is the solution to the following Poission equation on $S^{n-1}$.
\begin{equation*}
-\Delta_{S^{n-1}} \varphi (y) = \int_{R_1}^{R_2} f (\rho  y  ) \rho^{n-1}  \dd \rho ,
\end{equation*}
for all $y \in S^{n-1}$. Then, it follows that the vector field $v$ satisfies the following properties:
\begin{itemize}
\item $v \in C_c^{l+1}\big ( T\mathbb{R}^n \big )$
\item $\supp v \subset \overline{A( R_1, R_2 )}$.
\item $v \big |_{\overline{A( R_1, R_2 )}}$ is real analytic on $\overline{A( R_1 , R_2 )}$.
\item $\dv_{\mathbb{R}^n} v = f$ holds on $\mathbb{R}^n$.
\end{itemize}
For the case of $n=2$, the same conclusions hold as well for the vector field $v$ on $\mathbb{R}^2$ which is given by:
\begin{equation}\label{verygoodvectorfieldnequals2}
v \big |_x = - \Big ( *_{\mathbb{R}^2} \Big \{ \int_{0}^1 f(\tau x ) \tau \dd \tau \textbf{r}^*(\Vol_{S^{1}} ) \big |_x
- \dd_{\mathbb{R}^2} \Big ( \chi \textbf{r}^* h \Big ) \big |_x  \Big \} \Big )_{\sharp} ,
\end{equation}
for all $x \in \mathbb{R}^2$, where $\chi$ is given by \eqref{chiGoodGood}, and $h : S^1 \rightarrow \mathbb{R}$ is the real analytic function on $S^1$ which satisfies $\int_{S^1} h \Vol_{S^1} = 0$ and the equation
\begin{equation*}
\dd_{S^1} h \big |_y =  \int_{R_1}^{R_2} f (\rho  y  ) \rho  \dd \rho \Vol_{S^1} \big |_y ,
\end{equation*}
for all $y \in S^1$.
\end{cor}

\newpage
\section{The main Theorem of our paper.}\label{mainthm}

In light of Lemma \ref{lem:s.4}, we consider the following rational functions in $\mu = \frac{R_2}{R_1}- 1 > 0$.

\begin{equation}\label{rationalZero}
\begin{split}
\mathcal{Q}_{0,1} (\mu )& = \frac{n\mu^{n+2} + \big ( \sum_{\alpha=2}^{n+1} \binom{n+2}{\alpha} \frac{n^2+3n}{\alpha+1} \mu^{\alpha} \big ) + \frac{(n+2)(n+1)n}{2} \mu }{\sum_{\alpha=2}^{n+2} \binom{n+2}{\alpha} (\alpha-2) \mu^{\alpha}} \\
\mathcal{Q}_{0,2} (\mu ) & = \frac{\sum_{\beta=0}^{n+1} \beta n \binom{n+2}{\beta +1} \mu^{\beta} }{\sum_{\alpha=2}^{n+2} \binom{n+2}{\alpha} (\alpha-2) \mu^{\alpha}}
\end{split}
\end{equation}

\begin{equation}\label{rationalOne}
\begin{split}
\mathcal{Q}_{1,1} (\mu ) & = \frac{ (n+2)(n+1)n \mu + 2(n+1)\sum_{\alpha =3}^{n+2} \binom{n+2}{\alpha} \mu^{\alpha-1}  }{\sum_{\alpha=2}^{n+2} \binom{n+2}{\alpha} (\alpha-2) \mu^{\alpha}} \\
\mathcal{Q}_{1,2} (\mu ) & = \frac{\sum_{\alpha=1}^{n+2} \binom{n+2}{\alpha} (\alpha -1) (2n+2-\alpha ) \mu^{\alpha-1} }{\sum_{\alpha=2}^{n+2} \binom{n+2}{\alpha} (\alpha-2) \mu^{\alpha}}
\end{split}
\end{equation}

\begin{equation}\label{rationalTwo}
\begin{split}
\mathcal{Q}_{2,1} (\mu ) & = \frac{(n+2) \sum_{\beta =1}^n \binom{n+1}{\beta + 1} \mu^{\beta} }{ \sum_{\alpha=2}^{n+2} \binom{n+2}{\alpha} (\alpha-2) \mu^{\alpha}} \\
\mathcal{Q}_{2,2}  (\mu ) & = \frac{(n+2) \sum_{\beta =0}^n \binom{n+1}{\beta + 1} \beta \mu^{\beta} }{ \sum_{\alpha=2}^{n+2} \binom{n+2}{\alpha} (\alpha-2) \mu^{\alpha}}
\end{split}
\end{equation}

With the above preparation, we can now state and prove the main theorem of this paper as follows:

\begin{thm}\label{TheMainTheoremWehave}
Consider $n \in \mathbb{Z}^+$ with $n \geq 2$, and $0< R_1 < R_2 < \infty$. Let $f : \overline{A(R_1, R_2)}\rightarrow \mathbb{R}$ be real analytic on $\overline{A(R_1 , R_2)}$ which satisfies $\int_{\mathbb{R}^n} f \Vol_{\mathbb{R}^n} = 0$. We now define $\widetilde{f} \in C_c^0(\mathbb{R}^n)$ as follows:
\begin{itemize}
\item
  $\widetilde{f} (\rho y ) = f (\rho y ) - \sum_{k=0}^2 \big ( \frac{-1}{R_1} \big )^k \big \{ \mathcal{Q}_{k,1} \big ( \frac{R_2 - R_1}{R_1} \big ) f(R_1 y) + \mathcal{Q}_{k,2} \big (\frac{R_2 -R_1}{R_1} \big ) f(R_2 y) \big \} \rho^k$ , for all $\rho \in [R_1 , R_2]$ and all $y \in S^{n-1}$. Here, $Q_{k,a}(\mu)$ are the rational functions in $\mu = \frac{R_2}{R_1}-1$ as given by \eqref{rationalZero}, \eqref{rationalOne} and \eqref{rationalTwo}.
\item For any $\rho \in (0,R_1) \cup (R_2 , \infty )$ and any $y \in S^{n-1}$, we define $\widetilde{f}(\rho y ) = 0$.
 \end{itemize}
In the case of $n \geq 3$, consider now the vector field $\textbf{U} : \overline{A(R_1 , R_2)} \rightarrow \mathbb{R}^n$, which is defined as follows:
\begin{equation}\label{veryKeyexpressionWinWin}
\begin{split}
 \textbf{U} \big |_x
& =  \big ( \int_{0}^1 \widetilde{f}(\tau x)\tau^{n-1} \dd \tau \big ) \sum_{k=1}^n x^k \frac{\partial}{\partial x^k} \Big |_x
+(-1)^{3n} \Big ( *_{\mathbb{R}^n} \dd_{\mathbb{R}^n} \big (  \chi \textbf{r}^* \big ( \dd_{S^{n-1}}^* (\widetilde{\varphi} \Vol_{S^{n-1}}) \big )  \big ) \Big )_{\sharp} \Big |_x \\
& + \frac{1}{|x|^n} \sum_{k=0}^2 \Big ( \frac{-1}{R_1} \Big )^k \Big \{ \sum_{a=1}^2 \mathcal{Q}_{k,a} \big (\frac{R_2 - R_1}{R_1} \big ) f\big (R_a \frac{x}{|x|} \big )     \Big \} \Big (\frac{|x|^{n+k}-R_1^{n+k}}{n+k} \Big )\sum_{k=1}^n x^k \frac{\partial}{\partial x^k} \Big |_x ,
\end{split}
\end{equation}
for all $x \in \overline{A(R_1 , R_2 )}$, where
\begin{equation}\label{cutoffanotherTWO}
\chi ( x ) =  \sin^{l+3} \Big ( \frac{\pi}{2 ( R_2 - R_1  )} (|x| - R_1 ) \Big ) \mathbf{1}_{[R_1, R_2]} (|x|) + \mathbf{1}_{[ R_2, \infty  )} (|x|) ,
\end{equation}

and $\widetilde{\varphi} : S^{n-1} \rightarrow \mathbb{R}$ is the solution to the following surface Poission equation on $S^{n-1}$ with $\int_{S^{n-1}} \widetilde{\varphi}  \Vol_{S^{n-1}} = 0$.
\begin{equation}\label{surfacePoissionFinal}
\begin{split}
- &\Delta_{S^{n-1}} \widetilde{\varphi} (y)\\
  = &\int_{R_1}^{R_2} \widetilde{f} (\rho y ) \rho^{n-1} \dd \rho \\
 =& \int_{R_1}^{R_2} f (\rho y ) \rho^{n-1} \dd \rho - \sum_{k=0}^2 \Big ( \frac{-1}{R_1} \Big )^k \Big ( \sum_{a=1}^2 Q_{k,a} \big ( \frac{R_2-R_1}{R_1}\big ) f(R_a y) \Big ) \Big ( \frac{R_2^{n+k}- R_1^{n+k}}{n+k} \Big ) ,
\end{split}
\end{equation}
for all $y \in S^{n-1}$. Then, it follows that $\textbf{U}$ as given by \eqref{veryKeyexpressionWinWin} is a real analytic vector field on
$\overline{A(R_1 , R_2 )}$ which satisfies $\dv_{\mathbb{R}^n} \textbf{U} = f$ on $\overline{A(R_1 , R_2)}$ and $\textbf{U} \big |_{\partial A(R_1, R_2)} = 0$.
In the case of $n=2$, the vector field $\textbf{U} : \overline{A(R_1 , R_2)} \rightarrow \mathbb{R}^2$ is defined by
\begin{equation}\label{veryKeyexpressionWinTWO}
\begin{split}
 \textbf{U}\big |_x
& =  \big ( \int_{0}^1 \widetilde{f}(\tau x)\tau^{n-1} \dd \tau \big ) \sum_{k=1}^2 x^k \frac{\partial}{\partial x^k} \Big |_x
+ \Big ( *_{\mathbb{R}^2} \dd_{\mathbb{R}^2} \big (  \chi \textbf{r}^* \widetilde{h} \Big )_{\sharp} \Big |_x \\
& + \frac{1}{|x|^2} \sum_{k=0}^2 \Big ( \frac{-1}{R_1} \Big )^k \Big \{ \sum_{a=1}^2 \mathcal{Q}_{k,a} \big (\frac{R_2 - R_1}{R_1} \big ) f\big (R_a \frac{x}{|x|} \big )     \Big \} \Big (\frac{|x|^{2+k}-R_1^{2+k}}{2+k} \Big )\sum_{k=1}^2 x^k \frac{\partial}{\partial x^k} \Big |_x ,
\end{split}
\end{equation}
for all $x \in \overline{A(R_1 , R_2 )}$, where $\widetilde{h} : S^1 \rightarrow \mathbb{R}$ is the real-analytic function on $S^1$ which satisfies
\begin{equation}
\dd_{S^1}\widetilde{h} \big |_y = \int_{R_1}^{R_2} \widetilde{f} (\rho y ) \rho \dd \rho \Vol_{S^1} \big |_y
\end{equation}
for all $y \in S^1$, and $\int_{S^1} h \Vol_{S^1} = 0$, and that $\chi$ is given by \eqref{cutoffanotherTWO}. Then, it follows that $\textbf{U}$ as given by \eqref{veryKeyexpressionWinTWO} is a real analytic vector field on
$\overline{A(R_1 , R_2 )}$ which satisfies $\dv_{\mathbb{R}^2} \textbf{U} = f$ on $\overline{A(R_1 , R_2)}$ and $\textbf{U} \big |_{\partial A(R_1, R_2)} = 0$.
\end{thm}

\begin{proof}
Here, we will give the proof of Theorem \ref{TheMainTheoremWehave} for the case of $n \geq 3$, since the case of $n=2$ is basically the same.
For $n \geq 3$, let $f : \overline{A(R_1 , R_2)}\rightarrow \mathbb{R}$ be real analytic on $\overline{A(R_1 ,R_2)}$, which satisfies
\begin{equation}\label{veryoriginalvanishing}
\int_{A(R_1, R_2)} f \Vol_{\mathbb{R}^n} = 0.
\end{equation}
As required in the hypothesis of theorem \ref{TheMainTheoremWehave}, define $\widetilde{f} : \mathbb{R}^n \rightarrow \mathbb{R}$ as follows:
\begin{equation}\label{widetildef}
\widetilde{f}(x) = \mathbf{1}_{[R_1, R_2]} (|x|)
\Big \{ f (x) - \sum_{k=0}^2 \big ( \frac{-1}{R_1} \big )^k
\Big ( \sum_{a=1}^2 \mathcal{Q}_{k,a} \big ( \frac{R_2 - R_1}{R_1} \big )
f(R_a \frac{x}{|x|} )  \Big ) |x|^k \Big \} ,
\end{equation}
for all $x \in \mathbb{R}^n$. It is plain to see that $\widetilde{f}$ as given by \eqref{widetildef} satisfies the following conditions:
\begin{itemize}
\item $\widetilde{f} \in C_c^0 (\mathbb{R}^n)$.
\item $\supp \widetilde{f} \subset \overline{A( R_1, R_2)}$ , where $A( R_1, R_2 ) =\{ x \in \mathbb{R}^n : R_1 < |x| < R_2 \}$.
\item $\widetilde{f} \Big |_{\overline{A( R_1, R_2 )}}$ is real analytic on $\overline{A( R_1, R_2 )}$.
\end{itemize}
Observe that Lemma \ref{lem:s.4} immediately gives the following relation:
\begin{equation}\label{goodintegrationvanishing}
\int_{A(R_1 , R_2 )} \sum_{k=0}^2 \big ( \frac{-1}{R_1} \big )^k
\Big ( \sum_{a=1}^2 \mathcal{Q}_{k,a} \big ( \frac{R_2 - R_1}{R_1} \big )
f(R_a \frac{x}{|x|} )  \Big ) |x|^k \Vol_{\mathbb{R}^n}\big |_x   = 0 .
\end{equation}

As such, \eqref{veryoriginalvanishing} together with \eqref{goodintegrationvanishing} implies that $\widetilde{f}$ as given by \eqref{widetildef} satisfies the following property
\begin{equation}
\int_{A(R_1, R_2)} \widetilde{f} \Vol_{\mathbb{R}^n} = 0 .
\end{equation}
we summarize what we see as above and conclude that $\widetilde{f}$ satisfies all the hypothesis of corollary \ref{corollaryfromCohomology} for the case of $l=0$. Thus, corollary \ref{corollaryfromCohomology} can now be applied to $\widetilde{f}$ directly.
To actually apply corollary \ref{corollaryfromCohomology}, we now consider the vector field $\widetilde{v} : \mathbb{R}^n \rightarrow \mathbb{R}^n$ which is defined as follows:
\begin{equation}\label{Widetildevectorfield}
\widetilde{v} \big |_x = (-1)^{3n+1} \Big ( *_{\mathbb{R}^n} \Big \{ \int_{0}^1 \widetilde{f}(\tau x ) \tau^{n-1} \dd \tau \textbf{r}^*(\Vol_{S^{n-1}} ) \big |_x
- \dd_{\mathbb{R}^n} \Big ( \chi \textbf{r}^* \big ( \dd_{S^{n-1}}^* ( \widetilde{\varphi} \Vol_{S^{n-1}}) \big )  \Big ) \big |_x  \Big \} \Big )_{\sharp} ,
\end{equation}
for all $x \in \mathbb{R}^n$, where the cut off function $\chi$ is given by \eqref{cutoffanotherTWO}, and $\widetilde{\varphi} : S^{n-1} \rightarrow \mathbb{R}$ is the solution to \eqref{surfacePoissionFinal} on $S^{n-1}$ with $\int_{S^{n-1}} \widetilde{\varphi} \Vol_{S^{n-1}} = 0$. Now, we apply Corollary \ref{corollaryfromCohomology}, in the case of $l=0$, to $\widetilde{v}$ to deduce that $\widetilde{v}$ as given by \eqref{Widetildevectorfield} must satisfy the following properties:
\begin{itemize}
\item $\widetilde{v} \in C_c^1\big ( T\mathbb{R}^n \big )$
\item $\supp \widetilde{v} \subset \overline{A( R_1, R_2 )}$.
\item $\widetilde{v} \big |_{\overline{A( R_1, R_2 )}}$ is real analytic on $\overline{A( R_1 , R_2 )}$.
\end{itemize}
and that
\begin{equation}\label{helpfuldivergence1}
\dv_{\mathbb{R}^n} \widetilde{v} = \widetilde{f} ,
\end{equation}
holds on $\mathbb{R}^n$.
Next, we consider the real analytic function $\lambda : \mathbb{R}^n -\{O\} \rightarrow \mathbb{R}$ defined as follows:
\begin{equation}\label{Forlambda}
\begin{split}
\lambda (x) & = \frac{1}{|x|^{n-1}} \int_{R_1}^{|x|} \sum_{k=0}^2 \sum_{a=1}^2 \Big ( \frac{-1}{R_1} \Big )^k
\mathcal{Q}_{k,a} \big ( \frac{R_2-R_1}{R_1} \big ) f (R_a \frac{x}{|x|}) \rho^{n+k-1} \dd \rho \\
& =  \frac{1}{|x|^{n-1}} \sum_{k=0}^2 \sum_{a=1}^2 \Big ( \frac{-1}{R_1} \Big )^k
\mathcal{Q}_{k,a} \big ( \frac{R_2-R_1}{R_1} \big ) f (R_a \frac{x}{|x|}) \Big ( \frac{|x|^{n+k}-R_1^{n+k}}{n+k} \Big ) ,
\end{split}
\end{equation}
for all $x \in \mathbb{R}^n-\{O\}$. Note that the vector field $\frac{\partial}{\partial \rho }$ on $\mathbb{R}^n-\{O\}$ can be expressed as follows:
\begin{equation*}
\frac{\partial}{\partial \rho } \Big |_x = \frac{1}{|x|} \sum_{k=1}^n x^k \frac{\partial}{\partial x^k} \Big |_x ,
\end{equation*}
for all $x \in \mathbb{R}^n -\{O\}$. Thus, it is clear that $\lambda \frac{\partial}{\partial \rho}$ is real analytic on $\mathbb{R}^n-\{O\}$. A direct computation gives the following relation:
\begin{equation}\label{helpfuldivergence2}
\dv_{\mathbb{R}^n} \Big ( \lambda \frac{\partial}{\partial \rho } \Big ) (x)
= \sum_{k=0}^2 \sum_{a=1}^2 \Big ( \frac{-1}{R_1} \Big )^k \mathcal{Q}_{k,a} \big ( \frac{R_2-R_1}{R_1} \big ) f \big ( R_a \frac{x}{|x|} \big ) |x|^k ,
\end{equation}
for all $x \in \mathbb{R}^n-\{O\}$. Next, we note that an application of Lemma \ref{lem:s.4} immediately gives the following relation
\begin{equation}\label{goodvanishing}
\lambda (R_2 y) = \frac{1}{R_2^{n-1}} \int_{R_1}^{R_2} \sum_{k=0}^2 \sum_{a=1}^2 \Big ( \frac{-1}{R_1} \Big )^k \mathcal{Q}_{k,a} \big ( \frac{R_2-R_1}{R_1} \big ) f \big ( R_a y \big ) \rho^{n+k-1} \dd \rho = 0,
\end{equation}
for all $y \in S^{n-1}$. Now, \eqref{Forlambda} and \eqref{goodvanishing} together gives
\begin{equation}\label{Vanishneeded1}
\lambda \big |_{\partial \big ( A(R_1 ,R_2) \big )} = 0 .\\
\end{equation}
Next, we define the vector field $\textbf{U} : \overline{A(R_1, R_2)} \rightarrow \mathbb{R}^3$ as follows:
\begin{equation}\label{GoodU}
\textbf{U} = \lambda \frac{\partial}{\partial \rho} + \widetilde{v} ,
\end{equation}
with $\lambda$ be given by \eqref{Forlambda}, and $\widetilde{v}$ be given by \eqref{Widetildevectorfield}.

Now, \eqref{widetildef}, together with \eqref{helpfuldivergence1} and \eqref{helpfuldivergence2} imply the following relation
\begin{equation}\label{resultONE}
\dv_{\mathbb{R}^n} \textbf{U} = f ,
\end{equation}
holds on $\overline{A(R_1 , R_2)}$.

Next, the facts that $\widetilde{v} \in C_c^0 (\mathbb{R}^n)$, $\supp \widetilde{v} \subset \overline{A(R_1 , R_2 )}$, and \eqref{Vanishneeded1} together immediately implies that the vector field $\textbf{U}$ as given by \eqref{GoodU} must satisfy
\begin{equation}\label{resultTWO}
\textbf{U} \big |_{\partial \big ( A(R_1 ,R_2) \big )} = 0 .
\end{equation}
Finally, the structure of \eqref{Widetildevectorfield} and that of \eqref{Forlambda} clearly indiates that $\textbf{U}$ is real analytic on
$\overline{A(R_1, R_2)}$. Finally, we note that:
\begin{itemize}
\item  by a straightforward computation, one can easily check that \eqref{GoodU} is equivalent to \eqref{veryKeyexpressionWinWin}. This completes the proof of Theorem \ref{TheMainTheoremWehave} for the case of $n \geq 2$.
 \end{itemize}
The argument for the case of $n=2$ is basically the same. As such, we omit the details here.
\end{proof}

\newpage
\section{Appendix: comparison of Lemma 2.4 with Spivak proof of Theorem $9$ of Chapter $8$ in \cite{Spivak}.}\label{Appendix}
\begin{lem}\label{basiclemmaCohomology}
$($see Theorem $9$ of Chapter 8 on page $364$ of \cite{Spivak}$)$ Let $n\in \mathbb{Z}^+$, then, it follows that
\begin{equation}
H_c^n \big ( \mathbb{R}^n\big ) = \mathbb{R} ,
\end{equation}
where $H_c^n \big ( \mathbb{R}^n\big )$ is defined to be the quotient vector space $\frac{\mathcal{Z}^n(\mathbb{R}^n)}{\mathcal{B}^{n-1} (\mathbb{R}^n)}$, with $\mathcal{Z}^n(\mathbb{R}^n)$ is the space of all smooth $n$-forms on $\mathbb{R}^n$ with compact supports, and $\mathcal{B}^{n} (\mathbb{R}^n) = \big \{ d \eta : \eta \in C^{\infty} ( \wedge^{n-1} T^*\mathbb{R}^n ) \big \}$.
\end{lem}

In particular, if we specifically focus on the details of the proof of Lemma \ref{basiclemmaCohomology} as given in Spivak book \cite{Spivak}, Spivak's method of proof of Theorem $9$ of Chapter $8$ in his book \cite{Spivak} gives exactly the following more precise result, which in turns gives Lemma \ref{basiclemmaCohomology} as a straightforward consequence.

\begin{lem}\label{Goodmotivation}
\textbf{$($Spivak's result and his method$)$}$($see the original argument by Spivak shown in page $364$ to page $367$ of \cite{Spivak}$)$
Let $n \in \mathbb{Z}^+$. Consider any real valued function $f \in C_c^{\infty} \big ( \mathbb{R}^n \big )$ which satisfies $\supp f \subset \overline{B_O(R)}$ for some $R> 0$ and that $\int_{B_O(R)} f \Vol_{\mathbb{R}^n} = 0$.
Define the real-valued function $g_1 : S^{n-1} \rightarrow \mathbb{R}$ on the standard unit sphere by
\begin{equation}\label{gONE}
g_1(y) = \int_{0}^{R} f (\rho y) \rho^{n-1} \dd \rho,
\end{equation}
for all $y \in S^{n-1}$. The fact
\begin{equation}
\int_{S^{n-1}} g_1 (y) \Vol_{S^{n-1}} = \int_{B_O(R)} f \Vol_{\mathbb{R}^n} = 0
\end{equation}
ensures the existence of some smooth $(n-2)$-form $\eta$ on $S^{n-1}$ for which
\begin{equation}\label{ugyeta}
g_1 \Vol_{S^{n-1}} = \dd_{S^{n-1}} \eta
\end{equation}
holds on $S^{n-1}$.
Define a smooth $(n-1)$-form $\Psi_1$ on $\mathbb{R}^n$ as follows.
\begin{equation}\label{PsiONE}
\begin{split}
\Psi_1 \Big |_x = & \Big ( \int_0^{|x|} f \big ( t \frac{x}{|x|} \big )  t^{n-1} \dd t  \Big ) \frac{1}{|x|^n} \sum_{k=1}^n (-1)^{k-1} x^k \dd x^1 \wedge \dd x^2... \dd x^{k-1}\wedge \dd x^{k+1}...\wedge \dd x^n \\
& - \dd_{\mathbb{R}^3} \Big ( \rho_1 \mathbf{r}^* ( \eta ) \big ) \Big ) ,
\end{split}
\end{equation}
with $\mathbf{r} : \mathbb{R}^n -\{O\} \rightarrow S^{n-1}$ be the standard retraction map $\mathbf{r}(x) = \frac{x}{|x|}$, and $\rho_1 : \mathbb{R}^n \rightarrow \mathbb{R}$ is any smooth cut-off function satisfying the following constraint
\begin{itemize}
\item $0 \leq \rho_1 (x) \leq 1$ holds for all $x \in \mathbb{R}^n$ satisfying $\frac{R}{2} \leq |x| \leq R$.
\item $\rho_1 (x) =0$, for $x \in \mathbb{R}^n$ satisfying $|x| \leq \frac{R}{2}$.
\item $\rho_1 (x) =1$, for $x \in \mathbb{R}^n$ satisfying $|x| \geq  R$.
\end{itemize}
 then, it follows that $\Psi_1 \in C_c^{\infty} \Big (  \bigwedge^{n-1} T^*\mathbb{R}^n  \Big )$ satisfies $\supp \Psi_1 \subset \overline{B_O(R)}$ and $\dd \Psi_1 = f \Vol_{\mathbb{R}^n}$ holds on $\mathbb{R}^n$.
\end{lem}

\begin{remark}
It is important to note that: Spivak gave a rigorous proof $($see the proof of Theorem $9$ of Chapter 8 from pages $364$ to $367$ of \cite{Spivak}$)$ of all the results stated in Lemma \ref{Goodmotivation}, even though Spivak did not phrase his result in the format as demonstrated in Lemma \ref{Goodmotivation}. Instead, Spivak simply summarize his result in the loose form of Lemma \ref{basiclemmaCohomology}. The reason is probably that Spivak wants to stress the cohomological aspect of the subject matter, and the explicit form of the expression \eqref{PsiONE} is not of great interest to him for further development in his books. However, in this paper, we need to \emph{compare} Spivak result as stated in Lemma \ref{Goodmotivation} with our Lemma \ref{keylemma} , we have to first summarize the method of Spivak $($as well as highlighting all the essential steps of Spivak method$)$ in the form of Lemma \ref{Goodmotivation}.
\end{remark}

Here, we need a couples of remarks to clarify the meaning and significance of Lemma \ref{keylemma} for the development of this research proposal.

\begin{remark}\label{rmk:lastRemark}
By a brief inspection, the reader of this paper can see that the structure of the statement of Lemma \ref{keylemma} basically follows the line of thought as demonstrated in Spivak's result Lemma \ref{Goodmotivation}. However, there are \emph{two sharp differences} between Lemma \ref{Goodmotivation} and Lemma \ref{keylemma}.

\begin{itemize}
\item The first new idea which we provided in the establishment of Lemma \ref{keylemma} is the use of a solution $\varphi$ to the Poisson equation $-\Delta_{S^{n-1}} C = g_2$, which allows us to replace the very arbitrary smooth $(n-2)$-form $\eta$ as involved in relations \eqref{ugyeta} and \eqref{PsiONE} in Lemma \ref{Goodmotivation} by the much more precise real analytic $(n-2)$-form $\dd_{S^{n-1}}^* \big ( \varphi \Vol_{S^{n-1}} \big )$ as involved in \eqref{eq:14} of Lemma \ref{keylemma}.
\item The second new idea which we provided in Lemma \ref{keylemma} is the use of a new cut off function $\chi$ with the structure of $\chi \big |_{\overline{A(\pi , 2 \pi )}}$ coincides with the restriction of the elementary real analytic function ${ \sin^{l+3}\Big ( \frac{ \pi}{2(R_2-R-1)} (|x|-R_1 )\Big ) }$ on $\overline{A(R_1 , R_2 )}$.
\end{itemize}
The above two points \emph{fully explains} the new contributions which we input into the establishment of Lemma \ref{keylemma}. Other than these two novel ingredients which we give, the logical structure of the proof of Lemma \ref{keylemma} essentially follows the one as given by Spivak in his proof for Lemma \ref{Goodmotivation}.
\end{remark}

\end{document}